\documentclass{article}

\usepackage{a4wide} %toggle wiaspreprint
\usepackage{authblk} %toggle wiaspreprint
\usepackage[dvipsnames]{xcolor} %toggle wiaspreprint

\usepackage[utf8]{inputenc}

\usepackage{latexsym,amssymb}
\usepackage{amsmath,amsthm,amsxtra,amsbsy,accents}
\usepackage{mathtools}
\usepackage[scr=boondoxo]{mathalpha}
\usepackage{bm}
\usepackage{hyperref}
\usepackage{subfig}

\usepackage{bbm}

\usepackage{enumitem}

\usepackage{upgreek}

\usepackage[only,llbracket,rrbracket]{stmaryrd}

\DeclareMathAlphabet{\mathup}{OT1}{\familydefault}{m}{n}
\newcommand{\dx}[1]{\mathop{}\!\mathup{d} #1}
\newcommand{\eps}{\varepsilon}
\DeclarePairedDelimiter{\abs}{\lvert}{\rvert}
\DeclarePairedDelimiter{\norm}{\lVert}{\rVert}
\DeclarePairedDelimiter{\bra}{(}{)}
\DeclarePairedDelimiter{\pra}{[}{]}
\DeclarePairedDelimiter{\set}{\{}{\}}
\DeclarePairedDelimiter{\skp}{\langle}{\rangle}
\makeatletter
\newcommand{\customlabel}[2]{%
   \protected@write \@auxout {}{\string \newlabel {#1}{{#2}{\thepage}{#2}{#1}{}} }%
   \hypertarget{#1}{}
}
\makeatother

\numberwithin{figure}{section}

\makeatletter
\newcommand{\dashover}[2][\mathop]{#1{\mathpalette\df@over{{\dashfill}{#2}}}}
\newcommand{\fillover}[2][\mathop]{#1{\mathpalette\df@over{{\solidfill}{#2}}}}
\newcommand{\df@over}[2]{\df@@over#1#2}
\newcommand\df@@over[3]{%
	\vbox{
		\offinterlineskip
		\ialign{##\cr
			#2{#1}\cr
			\noalign{\kern1pt}
			$\m@th#1#3$\cr
		}
	}%
}
\newcommand{\dashfill}[1]{%
	\kern-.5pt
	\xleaders\hbox{\kern.5pt\vrule height.4pt width \dash@width{#1}\kern.5pt}\hfill
	\kern-.5pt
}
\newcommand{\dash@width}[1]{%
	\ifx#1\displaystyle
	2pt
	\else
	\ifx#1\textstyle
	1.5pt
	\else
	\ifx#1\scriptstyle
	1.25pt
	\else
	\ifx#1\scriptscriptstyle
	1pt
	\fi
	\fi
	\fi
	\fi
}
\newcommand{\solidfill}[1]{\leaders\hrule\hfill}
\makeatother
 \def\calB{{\mathcal B}} 
\def\calD{{\mathcal D}} \def\calE{{\mathcal E}} 
  \def\calI{{\mathcal I}}
  
\def\calM{{\mathcal M}}  
\def\calP{{\mathcal P}}  \def\calR{{\mathcal R}}
\def\calS{{\mathcal S}}

\def\rmj{{\mathrm j}}  
\def\rmm{{\mathrm m}}

 \def\sfe{{\mathsf e}} 
\def\sfg{{\mathsf g}}  
  
 \def\sfn{{\mathsf n}} 
  
 \def\sft{{\mathsf t}} 
\def\sfv{{\mathsf v}} \def\sfw{{\mathsf w}}

 \def\sfE{{\mathsf E}} 
\def\sfG{{\mathsf G}}  
  \def\sfL{{\mathsf L}}

\def\sfV{{\mathsf V}}

  \def\scrL{{\mathscr  L}}

\newcommand{\R}{\mathbb{R}}
\newcommand{\N}{\mathbb{N}}
\newcommand{\bew}{{\rm BE}_w (c, \infty)}
\newcommand{\BE}[1]{{\rm BE} (#1, \infty)}

\newcommand{\CD}[1]{{\rm CD} (#1, \infty)}
\newcommand{\RCDw}{{\rm RCD}_w (c, \infty)}
\newcommand{\RCD}[1]{{\rm RCD} (#1, \infty)}
\newcommand{\EVIw}{{\rm EVI}_w(c)}
\newcommand{\EVI}[1]{{\rm EVI}_{#1}}
\newcommand{\Dom}{{\rm Dom}}

\newcommand{\Meas}{\calM}
\newcommand{\MeasPos}{\calM_{+}}
\newcommand{\Prob}{\calP}

\newcommand{\ProbAc}{\calP_{\rm ac}}
\newcommand{\leb}{\scrL}

\newcommand{\Ch}{{\rm{Ch}}}
\newcommand{\Ent}{\sfE\sfn\sft}

\newcommand{\AC}{{\rm AC}}
\newcommand{\lip}[1]{{\rm{lip}}_{#1}}
\newcommand{\Lip}[1]{{\rm{Lip}}_{#1}}

\newcommand{\setsep}{\,:\,}
\newcommand{\restrto}[1]{|_{#1}}

\newcommand{\normal}{\sfn}

\newcommand{\weak}[1]{\abs*{{\rm D} #1}_w}

\newcommand{\weakcvg}{\rightharpoonup}

\makeatletter
\newcommand{\bnabla}{{\mathpalette\b@nabla\relax}}
\newcommand\b@nabla[2]{%
        \setbox\z@=\hbox{$\m@th#1\bigtriangledown$}%
        \ht\z@.7\ht\z@
        \raise\dp\z@\box\z@
}
\makeatother

\newtheorem{theorem}{Theorem}[section]

\newtheorem{lemma}[theorem]{Lemma}
\newtheorem{corollary}[theorem]{Corollary}
\newtheorem{definition}[theorem]{Definition}

\newtheorem{proposition}[theorem]{Proposition}
\newtheorem{remark}[theorem]{Remark}
\newtheorem{remark*}{Remark}

\numberwithin{equation}{section}

\makeatletter
\newcommand{\oset}[3][0ex]{%
  \mathrel{\mathop{#3}\limits^{
    \vbox to#1{\kern-2\ex@
    \hbox{$\scriptstyle#2$}\vss}}}}
\makeatother
\usepackage{trimclip}
\makeatletter
\NewDocumentCommand{\rmjmath}{}{\mathbin{\mathpalette\eplus@\relax\mspace{1mu}}}
\newcommand{\eplus@}[2]{\clipbox{-.5 -.5 0 {.35\height}}{$\m@th#1\rmj$}}
\makeatother

\newcommand{\one}{1\!\!1}

\usepackage{tikz}
\usetikzlibrary{arrows.meta}
\usetikzlibrary{calc,intersections,through}

\definecolor{bluegray}{rgb}{0.4, 0.6, 0.8}

\title{Weak curvature conditions on metric graphs}
\author[1]{Juliane Krautz}
\affil[1]{{
\small
Universit\"{a}t Augsburg, Institut f\"ur Mathematik, Universit\"{a}tsstra\ss e 12a, 86159 Augsburg, Germany. Email: juliane.krautz@uni-a.de
}}
\date{\today}

\begin{document}

\maketitle

\begin{abstract}
    Starting from pointwise gradient estimates for the heat semigroup, we study three characterizations of weak lower curvature bounds on metric graphs.
    More precisely, we prove the equivalence between a weak notion of the Bakry-Émery curvature condition, a weak Evolutionary Variational Inequality and a weak form of geodesic convexity. The proof is based on a careful regularization of absolutely continuous curves together with an explicit representation of the Cheeger energy. 
    We conclude with a brief discussion on possible applications to the Schrödinger bridge problem on metric graphs.
\end{abstract}

\tableofcontents

\section{Introduction}

Over the past years, the study of synthetic curvature bounds on metric measure spaces and their relation has received considerable interest. 
In their seminal work, Lott-Villani \cite{lottvillani09} and Sturm \cite{sturm_geo1_06, sturm_geo2_06} introduced a new notion of curvature bound on these spaces using the displacement convexity of the logarithmic entropy. Let $(X,d,\rmm)$ be a metric measure space, that is a complete separable metric space $(X,d)$ endowed with a probability measure that has finite second moment, we write $\rmm\in\Prob_2(X)$. 
The logarithmic entropy on that space is given as
\begin{align}\label{eq:Entr_refmeas}
    \Ent_\rmm(\mu) = \int_X  f\log f \dx{\rmm} \quad \text{for}\quad \frac{\dx{\mu}}{\dx{\rmm}} =  f,
\end{align}
where $\frac{\dx{\mu}}{\dx{\rmm}}$ is the Radon-Nikodym density of $\mu$ with respect to $\rmm$. Let
\begin{align}\label{eq:DefW2}
    W^2_2(\mu_0,\mu_1) = \inf\limits_{\pi\in\Pi(\mu_0,\mu_1)} \iint_{X\times X} d^2(x,y) \dx{\pi(x,y)}
\end{align}
denote the $2$-Wasserstein distance for the set of transport plans
\begin{align}\label{eq:AdmPlans}
    \Pi(\mu_0,\mu_1) = \set*{\pi\in\Prob(X\times X) \setsep \pi(A\times X) = \mu_0(A), \, \pi(X\times B) = \mu_1(B), \, A,B\subset X \, \text{Borel}}
\end{align}
between $\mu_0\in \Prob_2(X)$ and $\mu_1\in\Prob_2(X)$. Indeed, $W_2$ defines a metric on $\Prob_2(X)$ and $(\Prob_2(X), W_2)$ is called Wasserstein space on $X$. The underlying space $(X,d,\rmm)$ is said to satisfy a $\CD{K}$-condition and to have curvature bounded from below by $K\in\R$ if the logarithmic entropy is geodesically $K$-convex, that is 
\begin{align}\label{eq:geo_cvx}
    \Ent_\rmm(\mu_s) \leq (1-s) \Ent_\rmm(\mu_0) + s \Ent_\rmm(\mu_1) - \frac{K}{2}s(1-s) W_2^2(\mu_0, \mu_1)
\end{align}
for all $s\in[0,1]$ and a Wasserstein geodesic $s\mapsto\mu_s$ connecting $\mu_0,\mu_1\in\Prob_2(X)$. 
Note that $\CD{K}$ is a special case of the Curvature-Dimension condition ${\rm CD}(K,N)$ introduced by \cite{sturm_geo2_06} and \cite{lottvillani09}. There, $N\in(0,\infty]$ is an upper bound on the dimension in addition to the lower curvature bound $K\in\R$. In the present paper, we restrict our analysis to the dimensionless case $N=\infty$.

If the underlying space is a Riemannian manifold $(M,g)$ and $\rmm = {\rm vol}_g$, the curvature bound given by $\CD{K}$ coincides with the geometric lower bound on the Ricci curvature of $M$. Since \eqref{eq:geo_cvx} is well-defined on metric measure spaces and does not depend on any Riemannian structure, $\CD{K}$ can be understood as an extension of this bound. 
However, it is possible for Finsler geometries to satisfy the $\CD{K}$-condition as well. This observation led to a refinement in \cite{ambrosio_MMS_14}, namely the introduction of spaces with Riemannian curvature bounded below satisfying a Riemannian Curvature-Dimension condition $\RCD{K}$.
In the same work, several equivalent characterizations of such spaces have been given, one of them by introducing the additional assumption of linearity of the heat flow, which is defined as the $L^2(X,\rmm)$-gradient flow of the Cheeger energy
\begin{align*}
    \Ch(f) = \frac{1}{2} \inf\set*{\liminf\limits_{n\to\infty} \int_X \abs{\nabla f_n}^2 \dx{\rmm} \setsep f_n~\text{is Lipschitz},\, \int_X\abs{f_n-f}^2\dx{\rmm} \to 0}
\end{align*}
and we write $\Dom(\Ch) = \set{f\in L^2(X) \setsep \Ch(f)\in\R}$ for its domain.
Linearity of this flow is equivalent to the Cheeger energy being a Dirichlet form, thus quadratic. Let $(P_t)_{t\geq0}$ denote the corresponding semigroup. By duality, it can be extended to the space of probability measures inducing the semigroup $(H_t)_{t\geq0}$. Among the powerful calculus tools developed in \cite{Ambrosio_2013}, it has been shown that $(H_t)_{t\geq0}$ coincides with the $W_2$-gradient flow of the logarithmic entropy on $\CD{K}$-spaces. This observation is a key ingredient in the proof of the following equivalent characterization. A metric measure space is an $\RCD{K}$-space if the heat flow satisfies an Evolutionary Variational Inequality, meaning that for all $\mu_0,\mu_1\in\Prob_2(X)$ with $\Ent_\rmm(\mu_1)<\infty$ the $W_2$-gradient flow of the logarithmic entropy exists and it holds that 
\begin{align}\label{eq:EVI}\tag{${\rm EVI}_K$}
    \frac{\dx{}}{\dx{t}} \frac{1}{2} W_2^2(H_t\mu_1,\mu_0) + \frac{K}{2}W_2^2(H_t\mu_1,\mu_0) + \Ent_\rmm(H_t\mu_1)\leq \Ent_\rmm(\mu_0) 
\end{align}
for a.e. $t>0$. 
Again, this notion is consistent with the classical Riemannian case and we refer to \cite{ambrosio_MMS_14} for more details. 

Another approach to generalizing curvature bounds has been developed in \cite{bakry_diffusions_1985} and originates from the theory of Dirichlet forms. Let $(X,\calB,\rmm)$ be a measure space, for example a metric measure space $(X,d,\rmm)$ with Borel $\sigma$-algebra induced by the metric. Further, let $\calE$ be a Dirichlet form on $X$ with Markov semigroup $(S_t)_{t\geq 0}$ and generator $\Delta$. We define
\begin{align*}
    \Gamma(f,g) \coloneqq \frac{1}{2} \bra*{\Delta(fg) - \Delta f g - f\Delta g},
\end{align*}
which is related to the Dirichlet form by
\begin{align*}
    \calE(f,g) = \int_X \Gamma(f,g)\dx{\rmm}.
\end{align*}
To shorten the notation, we write $\Gamma(f) = \Gamma(f,f)$.
In the particular case of the measure space being a Riemannian manifold $(M,\sfg)$ with volume measure $\rmm$, we can choose $\Delta = \Delta_\sfg$ as the Laplace-Beltrami operator and obtain $\Gamma(f,g) = \skp{\nabla f, \nabla g}_\sfg$. 
Omitting the dependence on the dimension by choosing $N=\infty$, an important consequence of Ricci-curvature bounded from below is the point-wise gradient estimate
\begin{align}\label{eq:BE}\tag{$\rm BE$}
    \Gamma(S_tf) \leq e^{-Kt}S_t\Gamma(f)
\end{align}
$\rmm$-a.e. in $M$ for all $t\geq 0$ and $f\in\Dom(\calE)$. We say that the condition $\BE{K}$ is satisfied if\eqref{eq:BE} holds.
This estimate does not depend on second order operators and can be formulated on metric measure spaces instead of Riemannian manifolds as well. For this reason, the optimal constant $K\in\R$ in \eqref{eq:BE} can be thought of as a lower bound on the curvature. On a Riemannian manifold this notion coincides with $K\in\R$ being a lower bound on the Ricci curvature as was shown in \cite{von_renesse_transport_2005}.

If the Cheeger energy is a Dirichlet form, or equivalently, if $P_t$ is linear for all $t\geq 0$, condition \eqref{eq:BE} can be stated for $\calE = \Ch$ with $S_t = P_t$.
Then, owing to the work of \cite{kuwada_duality_2010}, the bound \eqref{eq:BE} is equivalent to the contraction estimate
\begin{align}\label{eq:ContrEst}
    W_2(H_t\mu, H_t\nu) \leq e^{-Kt}W_2(\mu,\nu)
\end{align}
for all $t\geq 0$ and $\mu,\nu\in\Prob_2(X)$.
Therefore, it is crucial to identify the Cheeger energy and study its properties when establishing weak notions of curvature bounds. Under the additional linearity  assumption \cite{ambrosio_bakryemery_2015, erbar_equivalence_2015} have shown that \eqref{eq:BE} is equivalent to the $\RCD{K}$-condition, thus providing the relation
\begin{align}\label{eq:classical_equiv}
    \BE{K} \Longleftrightarrow \EVI{K} \Longleftrightarrow \RCD{K}
\end{align}
and therefore connecting both approaches.
In view of \eqref{eq:classical_equiv}, we call a metric measures space an $\RCD{K}$-space if, for $K\in\R$, any of $\BE{K}$, $\EVI{K}$ or $\RCD{K}$ holds true. 

These conditions allow to conclude several important properties, such as functional and geometric inequalities \cite{lott_weak_2007, villani_optimal_2009} or $\Gamma$-convergence results \cite{monsaingeon_dynamical_2023}.
Besides $\RCD{K}$-spaces covering a broad class of examples, many widely studied spaces do not satisfy any of the equivalent conditions from \eqref{eq:classical_equiv}. Metric graphs and sub-Riemannian manifolds provide prototypical examples where curvature bounds fail to hold true, see \cite{Erbar_2022, baudoin_differential_2018} or \cite{ambrosio2020heat}.  
For this reason, the question of extending curvature conditions to a weak notion received interest in recent years. The following observation serves as a starting point for possible extensions. Several spaces not satisfying an $\RCD{K}$-condition still admit a linear heat-flow and point-wise gradient bounds of the form 
\begin{align}\label{eq:weak_gradEstX}
    \Gamma(P_tf) \leq C^2e^{-2Kt} P_t\Gamma(f) \quad\text{for all } t\geq 0 \text{ and } f\in \Dom(\Ch)
\end{align}
with constants $C\geq1$ and $K\in\R$, see for example \cite{baudoin_differential_2018} for metric graphs and \cite{DRIVER2005340} for Heisenberg groups. We write $\bew$ for the condition above.
If $C=1$, then \eqref{eq:weak_gradEstX} coincides with \eqref{eq:BE} and we can think of $c(t) = e^{-Kt}$ as a curvature function, where $K\in\R$ is the lower curvature bound. In the more general case $C\geq1$ we call $c(t) = Ce^{-Kt}$ the curvature function as well and, following this analogy, we define the optimal constant $K\in\R$ to be a weak lower bound on the curvature of $(X,d,\rmm)$. In the spirit of \cite{kuwada_duality_2010}, estimate \eqref{eq:weak_gradEstX} is equivalent to
\begin{align}\label{eq:ContrEstw}
    W_2(H_t\mu_1, H_t\mu_0) \leq Ce^{-Kt}W_2(\mu_1,\mu_0) \quad\text{for all }  t\geq 0 \text{ and } \mu_1,\mu_0\in\Prob_2(X)
\end{align}
with the dual heat semigroup $(H_t)_{t\geq0}$. 
Following the ideas of \cite{ambrosio_bakryemery_2015}, an equivalence result similar to \eqref{eq:classical_equiv} has been shown in \cite{stefani_generalized_2022} for metric measure groups.
Due to the equivalence \eqref{eq:classical_equiv}, these spaces do not admit a geodesically convex entropy and the weak condition extending $\RCD{K}$ under \eqref{eq:weak_gradEstX} originates in the idea of quantifying this non-convexity. However, if the underlying space exhibits branching geodesics, the entropy along them might be infinite, which corresponds to $\Ent_\rmm(\mu_s) = \infty$ for the left-hand side of \eqref{eq:geo_cvx}. For this reason, the heat semigroup is employed to regularize possible singularities.  
We write $\RCDw$ for the resulting condition, meaning that
\begin{align}
    \Ent_\rmm(H_{t+h}\mu_s) \leq (1-s) &\Ent_\rmm(H_t\mu_0) + s \Ent_\rmm(H_t\mu_1) \notag \\&+ \frac{s(1-s)}{2h} \bra*{\frac{1}{R(t, t+h)} W_2^2(\mu_0, \mu_1) - W_2^2(H_t\mu_0, H_t\mu_1)} \label{eq:geo_cvx_weak}
\end{align}
holds for all $t\geq 0$, $h>0$ and $s\mapsto \mu_s$ a $W_2$-geodesic joining $\mu_0,\mu_1\in\Prob_2(X)$ with 
\begin{align*}
    R(t_0,t_1) = \int_0^1 c^{-2}\bra*{(1-s)t_0 + s t_1} \dx{s}.
\end{align*}
For $t=0$, this inequality reduces to a heated version of \eqref{eq:geo_cvx}, namely
\begin{align}\label{eq:heatetd_geo_cvx}
    \Ent_\rmm(H_{h}\mu_s) \leq (1-s) &\Ent_\rmm(\mu_0) + s \Ent_\rmm(\mu_1) + \frac{s(1-s)}{2h} \bra*{\frac{1}{R(0, h)} - 1} W_2^2(\mu_0, \mu_1).
\end{align}
If the space satisfies an $\RCD{K}$-condition, we have $c(t) = e^{-Kt}$ and one can show that $\frac{1}{h}\bra*{\frac{1}{R(0, h)} - 1} \to -K$ as $h\to 0$, therefore recovering \eqref{eq:EVI}.
The weak condition $\EVIw$ admits another reformulation that characterizes the defect to geodesic convexity in more detail. Let $\calI:\Prob(X)\to[0,\infty]$ with $\calI(f) = 4 \Ch(\sqrt{f})$ be the Fisher Information.
If the chain rule
\begin{align}
    \Ent_\rmm(H_h\mu_s) - \Ent_\rmm(\mu_s) = -\int_0^h \calI(H_\tau\mu_s)\dx{\tau}
\end{align}
holds, we can rewrite \eqref{eq:heatetd_geo_cvx} as the distorted convexity condition
\begin{align*}
    \Ent_\rmm(\mu_s) \leq (1-s) &\Ent_\rmm(\mu_0) + s \Ent_\rmm(\mu_1) + \omega(s),
\end{align*}
where we define
\begin{align*}
    \omega(s) \coloneqq \frac{s(1-s)}{2h} \bra*{\frac{1}{R(0, h)} - 1} W_2^2(\mu_0, \mu_1) + \int_0^h \calI(H_\tau\mu_s)\dx{\tau}
\end{align*}
for fixed $h>0$ quantifying the non-convexity.
The weak version of \eqref{eq:EVI} that turns out to be equivalent to \eqref{eq:BEw} and \eqref{eq:geo_cvx_weak} can be understood as an almost integrated version of the Evolutionary Variational Inequality. We write $\EVIw$ to denote that
\begin{align*}
    \frac{1}{2} W_2^2(H_{t_1}\mu_1, H_{t_0}\mu_0) - \frac{1}{2 R(t_0,t_1)}W_2^2(\mu_1, \mu_0) \leq (t_1-t_0)\bra*{\Ent_\rmm(H_{t_0}\mu_0) - \Ent_\rmm(H_{t_1}\mu_1)}
\end{align*}
holds for $\mu_0, \mu_1 \in \Prob_2(X)$ and $0\leq t_0 \leq t_1 \leq 1$. For a detailed overview on related literature with a particular focus on curvature conditions in sub-Riemannian manifolds, we refer the reader to \cite{stefani_generalized_2022} and the references therein. With these weak notions, their result reads
\begin{align}\label{eq:equiv_weak}
    \bew \Longleftrightarrow \EVIw  \Longleftrightarrow \RCDw,
\end{align}
giving rise to the definition of metric measure spaces satisfying weak lower curvature bounds by one of the three equivalent conditions above. 

To the author's best knowledge, the equivalence is only known for special metric measure groups so far, even though the idea of studying possible extensions to spaces like metric graphs or the Grushin plane was already formulated in \cite{stefani_generalized_2022}. The proof relies on the construction of a regularization of Wasserstein geodesics compatible with metric derivatives. In $\RCD{K}$-spaces this regularization is obtained using the heat semigroup and compatibility is inferred from the contraction estimate \eqref{eq:ContrEst}. For the metric derivatives it implies 
\begin{align*}
    \abs{\Dot{H_t\mu}_s} \leq e^{-2Kt} \abs*{\Dot{\mu}_s} \leq \abs*{\Dot{\mu}_s}.
\end{align*}
If the weak condition \eqref{eq:ContrEstw} holds instead, we only obtain 
\begin{align*}
    \abs{\Dot{H_t\mu}_s} \leq C e^{-2Kt} \abs*{\Dot{\mu}_s} \leq C\abs*{\Dot{\mu}_s}
\end{align*}
up to the constant $C\geq1$. Taking the limit as $t\to 0$ does not recover the speed of the original curve as long as $C>1$ and we are not in the $\RCD{K}$-setting.
For this reason, it is necessary to introduce a new approximation given by a family of curves of measures $(\mu_s^n)_{n\in\N}$ compatible with metric derivatives in the sense that  
\begin{align}\label{eq:Compatible}
    \lim\limits_{n\to\infty} \abs*{\Dot{\mu}^n_s} \leq \abs*{\Dot{\mu}_s}
\end{align}
and such that $\mu^n_s \weakcvg \mu_s$.
In the case of metric measure groups, this is achieved by means of the operation of group convolution. As soon as we leave the group setting, this approximation is no longer applicable. Additionally the characterization of the gradient flow of $\Ent_\rmm$ as the heat equation has to be known a priori.
However, on metric graphs both issues have been resolved in \cite{Erbar_2022}. There, the authors introduced a regularization that satisfies \eqref{eq:Compatible} and showed that the heat equation can be understood as the $W_2$-gradient flow of $\Ent_\rmm$. Since metric graphs admit weak gradient bounds, see \cite{baudoin_differential_2018}, this observation provides the starting point of our analysis with the aim of extending \eqref{eq:equiv_weak}. \\

\textbf{Contributions.}
The main contribution of the present article is the proof of the equivalences from \eqref{eq:equiv_weak} on metric graphs. We rely on the regularization introduced in \cite{Erbar_2022} and adapt the arguments presented in the metric measure group setting, therefore extending \cite{stefani_generalized_2022}[Theorem 5.16]. As a first step, we show that the Cheeger energy coincides with the Dirichlet energy, allowing us to connect the abstract setting of metric measure spaces to the gradient bound from \cite{baudoin_differential_2018}. This identification relies on the identity $W^{1,\infty} = {\rm Lip}$ on metric graphs for which we give a proof as we could not find it in the literature. Here, ${\rm Lip}$ denotes the space of Lipschitz functions.
Moreover, we introduce the notion of strongly regular curves on metric graphs. For arbitrary absolutely continuous curves we then construct an approximating sequence of strongly regular curves. In particular, the action and entropy are well-behaved along weak limits. Employing this procedure together with a regularization of the entropy, we then derive an estimate similar to $\EVIw$. Starting from this estimate, we conclude the equivalences by showing that the inequality is preserved under taking the limit in the regularization. \\

\textbf{Structure of the Paper.}
The paper is structured as follows. 
In Section \ref{sec:Prelim} we define metric graphs as well as introducing the heat equation on such spaces. In Section \ref{sec:Heat} we study this equation and its two interpretations as a gradient flow of $\Ch$ and $\Ent_\rmm$. To this end, we provide an identification of the Cheeger energy in Theorem \ref{thm:IdCheeger}. Moreover, we introduce the regularization of measures on graphs applied throughout the article in Definition \ref{def:Regular}.
In Section \ref{sec:CurvatureBnds} we state and prove the equivalences in Theorem \ref{thm:equivalence}. As a first step, we show that \eqref{eq:weak_gradEstX} is equivalent to \eqref{eq:ContrEstw} in Theorem \ref{thm:kuwada_dual}. The proof relies on known results in metric measure spaces and the identification of the Cheeger energy from the previous section.
We introduce strongly regular curves and show existence of such in Theorem \ref{thm:ExStrongReg} employing the regularization from Definition \ref{def:Regular}.
Lastly, we establish action and entropy estimates along such curves, allowing us to conclude the proof.

\section{Preliminaries}\label{sec:Prelim}

Let $\sfG = (\sfE, \sfV)$ be a connected combinatorial graph with a finite set of vertices $\sfV$ and edges $\sfE\subset\sfV\times\sfV$.
We endow the graph with a map $\ell:\sfE \to (0,\infty)$, $\ell(\sfe) =: \ell_\sfe$ and define an orientation on each edge using the outer normal 
\begin{align*}
    \sfn^\sfe(\sfv) \coloneqq \begin{cases}
    -1 &: \sfe = (\sfv,\sfw) \\
    +1 &: \sfe = (\sfw,\sfv) \\
    0 &: \sfv\notin \sfe
\end{cases},
\end{align*}
allowing us to identify edges with intervals $[-\frac{\ell_\sfe}{2}, \frac{\ell_\sfe}{2}]$. We call the triple $\sfG=(\sfV, \sfE, \ell)$ a \emph{metric graph}.
For each $\sfv\in\sfV$ we further denote the set of incident edges by $\sfE(\sfv) \coloneqq \{\sfe\in\sfE \ | \sfv\in \sfe\}$ and define the spaces
\begin{align}\label{eq:DomGraph}
    \sfL \coloneqq \bigsqcup\limits_{\sfe\in\sfE} \pra*{-\frac{\ell_\sfe}{2}, \frac{\ell_\sfe}{2}} \quad \text{and} \quad \sfG \coloneqq \sfL/\sim,
\end{align}
where the first set is a collection of the one-dimensional intervals associated to the edges. In the second set $\sfG$ we identify vertices uniquely using the equivalence relation $\sim$. Together with the \emph{graph distance} $d: \sfG \times \sfG \to [0,\infty)$, which is the length of the shortest path between two points on $\sfG$, the tuple $(\sfG, d)$ defines a metric space.
We denote by $\Prob(\sfG)$ the space of probability measures on the metric graph, by $\Meas(\sfG)$ the space of Radon measures and by $\MeasPos(\sfG)$ the space of non-negative Radon measures. Additionally, let $\ProbAc(\sfG)$ be the set of probability measures that are absolutely continuous with respect to the Lebesgue measure $\leb$. Here, $\leb\in \MeasPos(\sfG)$ is the measure on $\sfG$ induced by the Lebesgue measure on $\sfL$ under the equivalence relation $\sim$. The corresponding Lebesgue spaces are denoted by $L^p(\sfG)$ for $p\in[1,\infty]$ with norm $\norm{\cdot}_{L^p(\sfG)}$. In particular, $(\sfG,d)$ is a geodesic space, rendering the Wasserstein space $(\Prob(\sfG), W_2)$ a geodesic space as well.

Let $(X,d)$ be an arbitrary metric space. For any closed interval $I\subset\R$ and $p\in[1,\infty]$ we call $\gamma:I\to X$ \emph{absolutely continuous}, and write $\gamma\in \AC^p\bra*{I; X}$, if there exists $g\in L^p(I)$ such that for all $t_0,t_1\in I$ with $t_0<t_1$ it holds that
\begin{align}\label{eq:DefAC}
    d\bra*{\gamma_{t_0}, \gamma_{t_1}} \leq \int_{t_0}^{t_1} g(\tau)\dx{\tau}.
\end{align}
By \cite{ambrosio_gradient_2008}[Theorem 1.1.2], every absolutely continuous curve $\gamma\in\AC^p\bra*{I;X}$ admits a \emph{metric derivative} $\abs*{\Dot{\gamma}_t}\in L^p(I)$ given as the smallest function $g\in L^p(I)$ satisfying \eqref{eq:DefAC}.
On a metric graph we define the \emph{local} and \emph{global Lipschitz constant} of a function $f:\sfG\to\R\cup\set{\infty}$ as 
\begin{align*}
    \lip{f}\bra{x} \coloneqq \limsup\limits_{y\to x} \frac{\abs*{f(x) - f(y)}}{d(x,y)} \quad \text{and} \quad \Lip{f} \coloneqq \sup\limits_{y\neq x } \frac{\abs*{f(x) - f(y)}}{d(x,y)}
\end{align*}
for $x\in\Dom(f) = \set*{x\in\sfG \setsep f(x)\in\R}$. 
Since $(\sfG,d)$ is a geodesic space, we have $\Lip{f} = \sup_{x\in\sfG} \lip{f}\bra{x}$.
The space of Lipschitz functions is denoted by ${\rm Lip}(\sfG)$, while $C(\sfG)$ is the space of continuous functions equipped with the supremum norm $\norm{\cdot}_{C(\sfG)}$. The space $C^1(\sfG)$ contains all functions from $C(\sfG)$ that are continuously differentiable on each edge of the graph and we denote the space of non-negative Lipschitz functions by ${\rm Lip}_{+}(\sfG)$.
In the same way, for $p\in[1,\infty]$ we define the Sobolev space $W^{1,p}(\sfG)$ as the space of functions in $L^p(\sfG)$ such that their weak derivative exists as a function in $L^p(\sfe)$ on each edge $\sfe\in\sfE$. As was shown in \cite{mugnolo_semigroup_2014}[Lemma 3.27], every element $f\in W^{1,p}(\sfG)$ admits a continuous representative in $C(\sfG)$ uniquely determined in the vertices.
Note that $\sfG$ is compact by definition, which is why all probability measures have finite second moment and why duality holds between $\Meas(\sfG)$ and $C(\sfG)$. We say that a sequence of measures $(\mu^n)_{n\in\N}\subset \Meas(\sfG)$ \emph{converges weakly} to another measure $\mu\in\Meas(\sfG)$, and write $\mu^n\weakcvg\mu$, if for all $\varphi\in C(\sfG)$ it holds that 
\begin{align*}
    \int_\sfG \varphi \dx{\mu^n} \to \int_\sfG \varphi\dx{\mu} \quad \text{as} \quad n\to\infty.
\end{align*}

Curvature bounds are closely related to the heat equation and the resulting semigroup on the underlying space. On metric graphs, this equation is understood in the following sense.
For each edge $\sfe\in\sfE$, we consider the one-dimensional heat equation on $[-\frac{\ell_\sfe}{2}, \frac{\ell_\sfe}{2}]$, coupling them in the vertices. Depending on the choice of coupling conditions, we obtain different systems. Throughout this article, we impose the so-called \emph{standard coupling conditions}. For $f^\sfe\coloneqq f\restrto{\sfe}$ they lead to
\begin{align}\label{eq:heat}
    \begin{cases}
        \partial_t f^\sfe(t,x) = \partial_{xx} f^\sfe(t,x) &t\geq 0, \, x\in\sfe,\, \sfe\in\sfE\\
        \sum_{\sfe\in\sfE(\sfv)} \partial_x f^\sfe(t,\sfv)\cdot \normal^\sfe(\sfv) = 0 &t\geq 0, \, \sfv\in\sfV\\
        f^\sfe(t,\sfv) = f^{\sfe'}(t,\sfv) &t\geq 0, \, \sfe,\sfe'\in\sfE(\sfv)
    \end{cases}
\end{align}
with initial condition $f(0,\cdot) = f_0(\cdot)$. 
Existence of solutions to \eqref{eq:heat} has been studied in \cite{fijavz_variational_2007} employing semigroup techniques and in \cite{Erbar_2022} using minimizing movements. 
As was first shown in \cite{roth_spectre_1984} and later extended to more general coupling conditions in \cite{kostrykin_heat_2007}, solutions to the heat equation can be characterized by a heat kernel $p_t(x,y)$. 
We collect some useful properties of the semigroup and its kernel below. 

\begin{lemma}\label{lem:heat_kernel}
    Let $\sfG=(\sfV,\sfE,\ell)$ be a metric graph and let $g_t(z) = \bra*{2\pi t}^{-\frac{1}{2}} e^{-\frac{z^2}{4t}}$ be the usual one-dimensional heat kernel. Then, there exists a heat kernel $p_t:\sfG\times\sfG\to\R$ for $t>0$ such that the following properties hold true. 
    \begin{enumerate}[label=\roman*)]
        \item For $T>0$ there exists a constant $C_{\rm up}(T)>0$ such that $0 \leq p_t(x,y) \leq C_{\rm up}(t) g_t(d(x,y))$ for all $0<t\leq T$ and $\leb$-a.e. $x,y\in\sfG$.
        \item The kernel is continuous. Moreover $\partial_t p_t(x,y) \in C(\sfG)$ for all $t>0$, $x,y\in\sfG$ and $\partial_x p_t(x,y) \in C(\sfL)$, $\partial_{xx} p_t(x,y) \in C(\sfL)$ for all $t>0$, $x,y\in\sfL$. 
        \item For all $t>0$ and $x\in \sfG$ it holds that $\norm{p_t(x,\cdot)}_{L^1(\sfG)} = 1$.
    \end{enumerate}
\end{lemma}

\begin{proof}
    A proof of the first part can be found in \cite{baudoin_differential_2018}[Lemma 5.3] and the last two statements have been shown in \cite{roth_spectre_1984}.
\end{proof}

\section{The heat equation as a gradient flow}\label{sec:Heat}

In this section, we study the heat equation on metric graphs and its two interpretations, namely the one as an $L^2$-gradient flow of the Cheeger energy and as a $W_2$-gradient flow of the logarithmic entropy. The latter has already been studied and made rigorous in \cite{Erbar_2022}, while the first one has been treated in general metric measure spaces in \cite{Ambrosio_2013}.
Our main goal is to show that both coincide on metric graphs, thus giving rise to the notion of heat flow. To this end, we show that the Cheeger energy is equivalent to the Dirichlet energy and that the abstract $L^2$-gradient flow induces the heat equation with standard coupling conditions from \eqref{eq:heat}. 

We begin by studying the Cheeger energy and its gradient flow.
Let $\Ch: L^2(\sfG)\to[0,\infty]$ with
\begin{align}\label{eq:Cheeger}
    \Ch(f) = \inf\set*{\liminf\limits_{n\to\infty} \int_\sfG \lip{f_n}\bra{x}^2 \dx{\leb} \setsep f_n\in {\rm Lip}(\sfG), ~f_n \to f \text{ in } L^2(\sfG)}
\end{align}
denote the \emph{Cheeger energy} and let the \emph{Dirichlet energy} $\calE: W^{1,2}(\sfG) \to [0,\infty)$ be given by
\begin{align*}
    \calE(f) = \int_\sfG \abs{\nabla f}^2\dx{x}.
\end{align*}
As was shown in \cite{Ambrosio_2013}[Theorems 6.2, 6.3] and \cite{Ambrosio_2013}[Remark 4.7], the Cheeger energy admits the integral representation
\begin{align}\label{eq:IntCheeger}
    \Ch(f) = \frac{1}{2}\int_\sfG \weak{f}^2 \dx{\leb}
\end{align}
for any $f\in\Dom(\Ch)$, where $\weak{f}$ is the so-called \emph{minimal relaxed gradient}. It is defined as the unique minimal element $g\in L^2(\sfG)$ such that there exists a family $(f_n)_{n\in\N}\subset {\rm Lip}(\sfG)$ with $f_n\to f$ and $\lip{f_n} \to g$ in $L^2(\sfG)$. It has been shown, see \cite{ambgig11}[Theorem 3.1] and \cite{ambrosio_gradient_2008} for the theory of Hilbertian gradient flows, that $\Ch$ gives rise to an $L^2$-gradient flow. Let $(P_t)_{t\geq0}$ denote the corresponding semigroup with dual $\bra*{H_t}_{t\geq0}$, extending the semigroup from $L^2(\sfG)$ to $\ProbAc(\sfG)$ by 
\begin{align*}
    H_t \mu = \bra*{P_t f}\leb \quad \text{if} \quad \frac{\dx{\mu}}{\dx{\leb}} = f.
\end{align*}
As a first step towards the characterization of $(H_t)_{t\geq0}$, we show that Lipschitz functions are weakly differentiable $\leb$-a.e. with bounded derivatives. More precisely, we prove that ${\rm Lip}(\sfG) = W^{1,\infty}(\sfG)$.
\begin{theorem}\label{thm:Rademacher}
    It holds that $f \in W^{1,\infty}(\sfG)$ if and only if $f\in {\rm Lip}(\sfG)$.
\end{theorem}

\begin{proof}
    First, let $f\in W^{1,\infty}(\sfG)$. Given \cite{mugnolo_semigroup_2014}[Lemma 3.27], there exists a continuous representative, which we again denote by $f\in C(\sfG)$ and it remains to show its Lipschitz-continuity. 
    Let $x,y\in\sfG$ be arbitrary points on the graph and suppose that there exists an edge $\sfe\in\sfE$ such that $x,y\in\sfe$. We denote by $f^\sfe$ the restriction of $f$ to $\sfe$, continuously extended by a constant value to a function on the real line. In particular $f^\sfe\in L^1_{\rm loc}(\R)$ by construction. For $n\in\N$ we define 
    \begin{align*}
        f^\sfe_n(x) \coloneqq \int_\R \zeta_n(x-y) f^\sfe(y) \dx{y},
    \end{align*}
    where $\zeta_n:\R\to\R$ is a standard mollifier such that $f_n^\sfe\to f^\sfe$ uniformly as $n\to\infty$ and $\norm*{\nabla f_n^\sfe}_{L^\infty(\R)} \leq \norm{\nabla f^\sfe}_{L^\infty(\R)}$. Choosing the parametrization of $\sfe$ such that $x\leq y$, we have
    \begin{align*}
        \abs*{f^\sfe_n(y) - f^\sfe_n(x)} &= \abs*{\int_x^y \nabla f^\sfe_n(z) \dx{z}} \leq \norm*{\nabla f^\sfe_n}_{L^\infty([x,y])}\abs{y-x} \leq \norm*{\nabla f^\sfe}_{L^\infty(\sfe)}d(x,y)
    \end{align*}
    and for $n\to\infty$ by uniform convergence 
    \begin{align}\label{eq:LipEdge}
        \abs*{f(y) - f(x)} \leq \norm*{\nabla f}_{L^\infty(\sfe)} d(x,y) \leq \norm*{\nabla f}_{L^\infty(\sfG)} d(x,y). 
    \end{align}
    Now, suppose that $x\in \sfe$ and $y\in \sfe'$ for $\sfe\neq \sfe'$. Since $(\sfG,d)$ is a geodesic metric space, we can find a shortest path connecting $x$ to $y$. We denote by $\bra*{v_i}_{i=0,\ldots,N}\subset \sfV$ the ordered sequence of vertices lying on that path, where $v_0 \coloneqq x$ and $v_N \coloneqq y$. Then 
    \begin{align*}
        \abs*{f(y) - f(x)} &\leq \sum\limits_{i=1}^N \abs*{f(v_i) - f(v_{i-1})} \leq \norm*{\nabla f}_{L^\infty(\sfG)} \sum\limits_{i=1}^N d(v_i,v_{i-1}) = \norm*{\nabla f}_{L^\infty(\sfG)} d(x,y),
    \end{align*}
    using \eqref{eq:LipEdge} in the second inequality. This shows $f\in {\rm Lip}(\sfG)$ with $\Lip{f} \leq \norm*{\nabla f}_{L^\infty(\sfG)}$.

    For the other implication, let $f\in {\rm Lip}(\sfG)$ be given. By definition $f\in C(\sfG)$, so that it remains to show weak differentiability on each edge. Fix $\sfe\in\sfE$ and let $f^\sfe:\R\to\R$ as before. For $\abs{\delta}\in (0,1)$ and $x\in \sfe$ we have that
    \begin{align*}
        \frac{f^\sfe(x) - f^\sfe(x-\delta)}{\delta} \leq \lip{f^\sfe}(x), 
    \end{align*}
    implying 
    \begin{align*}
        \norm*{\frac{f^\sfe(x) - f^\sfe(x-\delta)}{\delta}}_{L^\infty(\sfe)} \leq \Lip{f^\sfe}\quad\text{and}\quad\norm*{\frac{f^\sfe(x) - f^\sfe(x-\delta)}{\delta}}_{L^2(\sfe)} \leq \sqrt{\ell_\sfe}\Lip{f^\sfe}
    \end{align*}
    by Hölder's inequality. Therefore, there exists a subsequence $\delta_k\to 0$ as $k\to\infty$ such that $\frac{f^\sfe(\cdot)-f^\sfe(\cdot-\delta_k)}{\delta_k}$ converges weakly to a function $g\in L^2(\sfe)$. For all $\varphi\in C_c^\infty(\sfe)$ we obtain
    \begin{align*}
        \int_\sfe \nabla \varphi(x) f^\sfe(x) \dx{x} &= \lim\limits_{k\to \infty} \int_\sfe \frac{\varphi(x+\delta_k) - \varphi(x)}{\delta_k} f^\sfe(x) \dx{x}\\
        &= - \lim\limits_{k\to\infty} \int_\sfe \varphi(x) \frac{f^\sfe(x) - f^\sfe(x-\delta_k)}{\delta_k} \dx{x}\\
        &= -  \int_\sfe \varphi(x) g(x) \dx{x},
    \end{align*}
    showing that $f$ is weakly differentiable with weak derivative $g = \nabla f^\sfe$. Since $\sfe\in\sfE$ was arbitrary, we have shown $f\in W^{1,2}(\sfG)$.
    Now, let $\eps>0$ and fix $\sfe\in\sfE$. We define the set 
    \begin{align*}
        A_\eps \coloneqq \set*{x\in \sfe\setsep \abs{\nabla f(x)} \geq \Lip{f} + \eps}
    \end{align*}
    and weak convergence implies
    \begin{align*}
        \leb(A_\eps)\bra*{\Lip{f} + \eps} \leq \int_{A_\eps} \abs{\nabla f^\sfe(x)} \dx{x} \leq \liminf\limits_{k\to \infty} \int_{A_\eps} \abs*{\frac{f^\sfe(x) - f^\sfe(x-\delta_k)}{\delta_k}} \dx{x} \leq \leb(A_\eps)\Lip{f}.
    \end{align*}
    Therefore, $\leb(A_\eps)=0$ for all $\eps>0$ and we infer that $\norm*{\nabla f^\sfe}_{L^\infty(\sfe)} \leq \Lip{f^\sfe}$. Since the choice $\sfe\in\sfE$ was arbitrary, this proves $f\in W^{1,\infty}(\sfG)$. 
\end{proof} 

Recalling the definition of the Cheeger energy from \eqref{eq:Cheeger}, it is necessary to construct approximations by Lipschitz functions when characterizing elements in its domain.
To this end, we make use of the ideas presented in \cite{Erbar_2022}.
We begin by extending the metric graph, introducing one auxiliary edge for each vertex. 
Let $\sfe = (\sfv,\sfw)\in\sfE$ be an arbitrary edge of $\sfG$, parametrized in such a way that $\sfe = \pra*{-\frac{\ell_\sfe}{2}, \frac{\ell_\sfe}{2}}$. The new edges are identified with $\sfe_\sfv^{2\eps} = \pra*{-\frac{\ell_\sfe}{2} - 2\eps, -\frac{\ell_\sfe}{2}}$ and $\sfe_\sfw^{2\eps} = \pra*{\frac{\ell_\sfe}{2}, \frac{\ell_\sfe}{2} + 2\eps}$, see Figure \ref{fig:ExtGraph} for an illustration. From this extension, we obtain a new vertex of degree one denoted by $\sfv^{2\eps}$ for each $\sfv\in\sfV$. Note that we only add one new edge for each vertex, however, depending on the incident edge $\sfe\in\sfE(\sfv)$, we parametrize it differently.
The extended metric graph is denoted by $\sfG^{2\eps} = (\sfV^{2\eps}, \sfE^{2\eps}, \ell^{2\eps})$ with domains defined in analogy to \eqref{eq:DomGraph}. We introduce the following regularization, see \cite{Erbar_2022}[Definitions 3.8, 3.11].

\begin{figure}
    \centering
    \begin{tikzpicture}[scale=.6]
        \filldraw (0,0) circle (2pt);
        \node[below] at (0,0) {$\sfv$};
        \filldraw (4,0) circle (2pt);
        \node[below] at (4,0) {$\sfw$};
        \draw[thick] (0,0) -- node[below]{$\sfe$} (4,0);
        \draw (0,0) -- (-1.5,0) node[below]{$\sfe_\sfv^{2\eps}$};
        \draw (4,0) -- (5.5,0) node[below]{$\sfe_\sfw^{2\eps}$};
        \draw[gray] (0,0) -- (-1,-1.5);
        \draw[gray] (0,0) -- (-1,1.5);
        \draw[gray] (4,0) -- (5,1.5);
        \draw[gray] (4,0) -- (5,-1.5);
    \end{tikzpicture}
    \caption{Example of an extension of a metric graph for $\eps>0$}
    \label{fig:ExtGraph}
\end{figure}
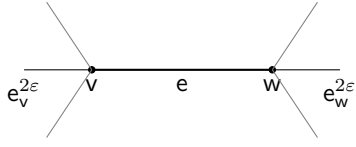

\begin{definition}\label{def:Regular}
    Let $\eps>0$ and $\varphi\in L^1(\sfG^{2\eps})$ be given. We define the function $\varphi^\eps:\sfG\to\R$ by 
    \begin{align*}
        \varphi^\eps (x) = \frac{1}{2\eps} \int_{\alpha^\eps_\sfe x - \eps}^{\alpha^\eps_\sfe x + \eps} \varphi(y) \dx{y}
    \end{align*}
    for $x\in \sfe$ and $\alpha^\eps_\sfe = \frac{\ell_\sfe + 2 \eps}{\ell_\sfe}$. For $\mu\in\Meas(\sfG)$ we define its regularization by duality as the measure $\mu^\eps\in\Meas(\sfG^{2\eps})$ such that for all $\varphi\in C(\sfG^{2\eps})$
    \begin{align*}
        \int_{\sfG^{2\eps}} \varphi \dx{\mu^\eps} = \int_\sfG \varphi^\eps \dx{\mu}.
    \end{align*}
\end{definition}

In general, we will deal with functions that are defined on $\sfG$ instead of the extended graph $\sfG^{2\eps}$ for some $\eps>0$. We use the following conventions. If $\varphi\in L^1(\sfG)$, we extend it by $\varphi\restrto{e^{2\eps}_\sfv} \equiv 0$ for all $\sfv\in \sfV$, leading to a function in $L^1(\sfG^{2\eps})$. On the other hand, if $\varphi\in C(\sfG)$, we extend it continuously by $\varphi\restrto{e^{2\eps}_\sfv} \equiv \varphi(\sfv)$ for $\sfv\in\sfV$.
This regularization satisfies several useful properties. 

\begin{proposition}\label{prop:PropertiesReg}
    Let $\eps>0$, $\varphi\in C(\sfG^{2\eps})$ and $\mu\in\Meas(\sfG)$ be given. The following hold true.
    \begin{enumerate}[label=\roman*)]
        \item For $\varphi^{\eps}:\sfG\to\R$ we have that $\varphi^{\eps}\in C^1(\sfG)$ with derivative given by
        \begin{align*}
            \nabla \varphi^{\eps}(x) = \frac{\alpha^\eps_\sfe}{2\eps}\bra*{\varphi(\alpha^\eps_\sfe x-\eps) - \varphi(\alpha^\eps_\sfe x+\eps)}
        \end{align*}
        for all $x\in \sfe$ and $\varphi^{\eps} \to \varphi$ uniformly in $\sfG$ as $\eps\to 0$. \label{RegCont}
        \item The regularization is mass preserving in the sense that $\mu^\eps(\sfG^{2\eps}) = \mu(\sfG)$. \label{MassCons}
        \item If $\mu\in\Prob(\sfG)$, then its regularization $\mu^\eps\in\Prob(\sfG^{2\eps})$ is absolutely continuous with respect to the Lebesgue measure and has the density $ f^\eps\in L^1(\sfG^{2\eps})$ given by
        \begin{align*}
             f^\eps(x) = \begin{cases}
                \frac{1}{2\eps} \mu(e\cap I_\sfe^\eps(x)) &: x\in \sfE\\ \frac{1}{2\eps}\bra*{\one_{\set{d(x,\sfv)\leq 2\eps}} \mu(\set{\sfv}) + \sum\limits_{\sfe\in\sfE(\sfv)}\mu(e\cap I_\sfe^\eps(x))} &: x\in\sfe^{2\eps}_\sfv, \, \sfv\in\sfV
            \end{cases}
        \end{align*}
        for 
        \begin{align*}
            I_e^\eps(x) = \bra*{\max\set*{\frac{x-\eps}{\alpha_\sfe^\eps}, -\frac{\ell_\sfe}{2}}, \min\set*{\frac{x+\eps}{\alpha_\sfe^\eps}, \frac{\ell_\sfe}{2}} }.
        \end{align*}
        Moreover it holds that $ f^\eps \leq \frac{1}{2\eps}$ on $\sfG^{2\eps}$. \label{density}
        \item As $\eps\to 0$ it holds that $\mu^\eps \weakcvg \mu$ in $\Prob(\sfG^{2\eps})$ for $\mu\in\Prob(\sfG)$. \label{cvg}
        \item If $s\mapsto\mu_s \in \AC^2\bra*{[0,1]; \Prob(\sfG)}$ then $s\mapsto \mu_s^\eps \in \AC^2\bra*{[0,1]; \Prob(\sfG^{2\eps})}$ and 
        \begin{align*}
            \abs*{\Dot{\mu}^\eps_s} \leq \bra*{1+\frac{2\eps}{\ell_{\min}}} \abs*{\Dot{\mu}_s}    
        \end{align*}
        for a.e. $s\in[0,1]$ with $\ell_{\min} \coloneqq \min\set*{\ell_\sfe \setsep \sfe\in\sfE}$. As a consequence
        \begin{align*}
            \limsup\limits_{\eps>0} \abs*{\Dot{\mu}^\eps_s} \leq \abs*{\Dot{\mu}_s}
        \end{align*}
        for a.e. $s\in[0,1]$. \label{ActionEst}
    \end{enumerate}
\end{proposition}

\begin{proof}
    The first result \ref{RegCont} is the content of \cite{Erbar_2022}[Proposition 3.9].
    The statements \ref{MassCons}, \ref{density} and \ref{cvg} are shown in \cite{Erbar_2022}[Proposition 3.12], where the convergence in \ref{cvg} has to be understood in the sense that 
    \begin{align*}
        \abs*{\int_{\sfG^{2\eps}} \varphi \dx{\mu^\eps} - \int_\sfG \varphi \dx{\mu}} = \abs*{\int_{\sfG} \varphi^{\eps} - \varphi \dx{\mu}} \to 0
    \end{align*}
    as $\eps\to 0$ for all $\varphi\in C(\sfG^{2\eps})$. For \ref{MassCons} note that the regularization of the constant function equal to one on the graph $\sfG^{2\eps}$ is again the constant function equal to one on the initial graph $\sfG$.
    The last statement follows from the same proposition together with \cite{Erbar_2022}[Theorem 3.7] which characterizes absolutely continuous curves as weak solutions to the continuity equation on metric graphs. 
\end{proof}

As a consequence of the last property, the regularization introduced in Definition \ref{def:Regular} is compatible with metric derivatives in the sense that it preserves absolute continuity and its metric derivative is asymptotically bounded by the original speed. Moreover, it provides sufficient regularity to construct Lipschitz approximations.

\begin{theorem}\label{thm:ApproxByLip}
    For every function $f\in W^{1,2}(\sfG)$ there exists a sequence of functions $f_n\in {\rm Lip}(\sfG)$, $n\in\N$, such that $f_n\to f$ uniformly on $\sfG$ as $n\to\infty$ with $\lim_{n\to\infty}\int_\sfG \lip{f_n}^2\dx{x} \leq \norm*{\nabla f}_{L^2(\sfG)}^2$.
\end{theorem}

\begin{proof}
    Let $f\in W^{1,2}(\sfG)$ be given. In particular, $f\in C(\sfG)$ by \cite{mugnolo_semigroup_2014}[Lemma 3.27] and, for each $\eps>0$, we find a function $ f^\eps\in C^1(\sfG)$ using Definition \ref{def:Regular} such that $ f^\eps \to f$ uniformly as $\eps\to0$ and $ f^\eps\in {\rm Lip}(\sfG)$ by Theorem \ref{thm:Rademacher}.
    For $x\in\sfG\setminus{\sfV}$ and any sequence $\bra*{y_n}_{n\in\N}\subset\sfG$ such that $d(x,y_n)\to 0$, we have $y_n\in\sfe$ for $n\in\N$ sufficiently large. Then, by the fundamental theorem of calculus and Proposition \ref{prop:PropertiesReg} \ref{RegCont} it holds that
    \begin{align*}
        \abs*{\frac{ f^\eps(x) -  f^\eps(y_n)}{d(x,y_n)}} &= \abs*{\frac{1}{\abs*{x-y_n}}\int_x^{y_n} \nabla  f^\eps(z) \dx{z}} 
        \leq \frac{\alpha^\eps_\sfe}{2\eps} \frac{1}{\abs*{x-y_n}}  \int_x^{y_n} \abs*{f\bra*{\alpha^\eps_\sfe z - \eps} - f\bra*{\alpha^\eps_\sfe z + \eps} } \dx{z} .
    \end{align*}
    This gives
    \begin{align*}
        \lip{ f^\eps}(x) \leq \alpha^\eps_\sfe \abs*{\frac{f(\alpha^\eps_\sfe x - \eps) - f(\alpha^\eps_\sfe x + \eps)}{2\eps}}
    \end{align*}
    since $f$ is continuous and every point is a Lebesgue point. Taking the squared $L^2(\sfe)$-norm of this estimate, we infer that
    \begin{align*}
        \int_\sfe \lip{ f^\eps}(x)^2 \dx{x} &\leq \bra*{\alpha^\eps_\sfe}^2 \int_\sfe \abs*{\frac{f(\alpha^\eps_\sfe x - \eps) - f(\alpha^\eps_\sfe x + \eps)}{2\eps}}^2 \dx{x}.
    \end{align*}
    By weak differentiability, we can apply \cite{brezis}[Proposition 9.3] and obtain the limit
    \begin{align*}
        \lim\limits_{\eps\to 0 }  \int_\sfe \lip{ f^\eps}(x)^2 \dx{x} &\leq \int_\sfe~ \abs*{\nabla f(x)}^2 \dx{x}.
    \end{align*}
    Summing over all edges and defining $f_n \coloneqq f^{\frac{1}{n}}$ concludes the proof.
\end{proof}

Having established the approximation result, we can identify Cheeger and Dirichlet energy.

\begin{theorem}\label{thm:IdCheeger}
    On a metric graph it holds that $\Dom(\Ch) = W^{1,2}(\sfG)$, $\abs*{\nabla f} = \weak{f}$ $\leb$-a.e. and
    \begin{align*}
        \calE(f) = 2\Ch(f).
    \end{align*}
\end{theorem}

\begin{proof}
    We split the proof into two parts, showing the inclusions separately. 
    
    First, we prove that for $f\in W^{1,2}(\sfG)$ we have $f\in \Dom(\Ch)$ with $2\Ch(f) \leq \norm*{\nabla f}^2_{L^2(\sfG)}$.
    Let $f\in W^{1,2}(\sfG)$ be given. As shown in Theorem \ref{thm:ApproxByLip}, there exists a sequence $f_n\in {\rm Lip}(\sfG)$ such that $f_n\to f$ in $C(\sfG)$ with 
        \begin{align*}
           \lim\limits_{n\to\infty} \int_\sfG \lip{f_n}(x)^2 \dx{x} \leq \norm*{\nabla f}_{L^2(\sfG)}^2.
        \end{align*}
        This directly implies
        \begin{align*}
            2\Ch(f) \leq \liminf\limits_{n\to\infty} \int_\sfG \lip{f_n}(x)^2 \dx{x} \leq  \norm*{\nabla f}_{L^2(\sfG)}^2 < \infty,
        \end{align*}
        so that $f\in\Dom(\Ch)$ and the bound holds true.

        Next, we show that $f\in\Dom(\Ch)$ implies $f\in W^{1,2}(\sfG)$ with $\norm*{\nabla f}^2_{L^2(\sfG)} \leq 2\Ch(f)$.
        Using the equivalent characterization of the Cheeger energy given in \eqref{eq:IntCheeger}, we know that for any $\eps>0$ there exists a family of functions $\bra*{f_n}_{n\in\N}\in {\rm Lip}(\sfG)$ such that $f_n\to f$ and $\lip{f_n} \to \weak{f}$ in $L^2(\sfG)$ with 
        \begin{align*}
            \lim\limits_{n\to\infty} \norm*{\lip{f_n}}_{L^2(\sfG)}^2 \leq 2\Ch(f) + \eps.
        \end{align*}
        In particular, for $n\in\N$ large enough
        \begin{align*}
            \norm*{\lip{f_n}}_{L^2(\sfG)}^2 \leq 2\Ch(f) + 2\eps.
        \end{align*}
        On each edge $\sfe = [-\frac{\ell_\sfe}{2}, \frac{\ell_\sfe}{2}]\in\sfE$ we extend every function $f_n\in {\rm Lip}(\sfG)$ by the constant value attained in the vertices, leading to a Lipschitz continuous extension. For $\abs{\delta}\in(0, 1)$ arbitrary we then have that 
        \begin{align*}
            \int_{-\frac{\ell_\sfe}{2} }^{\frac{\ell_\sfe}{2}} \abs*{\frac{f_n(x+\delta) - f_n(x)}{\delta}}^2 \dx{x} \leq  \int_{-\frac{\ell_\sfe}{2}}^{\frac{\ell_\sfe}{2}} \lip{f_n}^2(x) \dx{x}.
        \end{align*}
        Observing that the right-hand side does not depend on $\abs{\delta}\in(0, 1)$, we can apply \cite{evans_partial_2010}[Chapter 5.8, Theorem 3] yielding weak differentiability on each edge and after summation over all edges this gives
        \begin{align}\label{eq:BndWeakDeriv}
            \norm*{\nabla f_n}^2_{L^2\bra*{\sfG}} \leq \sum\limits_{\sfe\in\sfE}\int_{-\frac{\ell_\sfe}{2}}^{\frac{\ell_\sfe}{2}} \lip{f_n}^2(x) \dx{x} \leq 2\Ch(f) + 2\eps.
        \end{align}
        Therefore, there exists a weakly convergent subsequence such that $\nabla f_{n_k} \weakcvg g$ in $L^2(\sfL)$ as $k\to\infty$. 
        However, for $\varphi\in C^1(\sfG)$ with $\varphi(\sfv)=0$ for all $\sfv\in\sfV$ we have that
        \begin{align*}
            \int_\sfG \nabla \varphi f \dx{x} \longleftarrow \int_\sfG \nabla \varphi f_{n_k} \dx{x} &= - \int_\sfG \varphi \nabla f_{n_k} \dx{x} \longrightarrow \int_\sfG \varphi g \dx{x}
        \end{align*}
        as $k\to\infty$, implying $g = \nabla f \in L^2(\sfL)$. This gives $\norm*{\nabla f}^2_{L^2(\sfG)} \leq 2\Ch(f) + 2\eps$ and, since $\eps>0$ was arbitrary, we have shown that $f\in W^{1,2}(\sfG)$ with 
        \begin{align*}
            \int_\sfG \abs*{\nabla f(x)}^2 \dx{x} \leq 2\Ch(f) = \int_\sfG \weak{f}^2(x) \dx{x}. 
        \end{align*}
    Due to uniqueness of the minimal relaxed gradient, we also obtain $\abs*{\nabla f} = \weak{f}$ $\leb$-a.e. in $\sfG$. 
\end{proof}

Employing Theorem \ref{thm:IdCheeger}, we are now in a position to characterize the $L^2$-gradient flow of $\Ch$ as the heat equation \eqref{eq:heat}.

\begin{lemma}
    On a metric graph the $L^2(\sfG)$-gradient flow of the Cheeger energy exists and is given by the heat equation \eqref{eq:heat}. The corresponding semigroup can be represented by the heat kernel from Lemma \ref{lem:heat_kernel}.
\end{lemma}

\begin{proof}
    The existence follows from \cite{ambgig11}[Theorem 3.1], whereas the representation results from Theorem \ref{thm:IdCheeger}.
\end{proof}

As another consequence of Theorem \ref{thm:IdCheeger}, the $\Gamma$-operator corresponding to the Cheeger energy admits the form
\begin{align*}
    \Gamma(f,g) = \nabla f \cdot \nabla g
\end{align*}
for $f,g\in W^{1,2}(\sfG)$ and we set $\Gamma(f) = \weak{f}^2 = \abs*{\nabla f}^2 $ for $f\in \Dom(\Ch) = W^{1,2}(\sfG)$.
As was shown in \cite{baudoin_differential_2018}[Theorems 5.4, 5.5], this operator satisfies a generalized Bakry-Émery estimate in the following sense. 
Let $\deg_{\max} = \max\set*{\deg(\sfv)\setsep\sfv\in\sfV}$. There exist constants $C\geq \deg_{\max}-1$ and $K>0$ such that for all $f\in W^{1,2}(\sfG)$ it holds that
\begin{align*}
    \sqrt{\Gamma(P_tf)} \leq Ce^{-Kt} P_t\sqrt{\Gamma(f)} ,
\end{align*}
which has to be understood as
\begin{align*}
    \sqrt{\abs*{\nabla \int_\sfG p_t(x,y) f(y) \dx{y}}^2} \leq Ce^{-Kt} \int_\sfG p_t(x,y) \sqrt{\abs*{\nabla f(y)}^2}\dx{y}
\end{align*}
for $\leb$-a.e. $x\in\sfG$ and $t\geq 0$. Further, it follows from the proof of this result that $K=\lambda_1$, where $\lambda_1>0$ is the smallest non-zero eigenvalue of the Laplacian with standard boundary conditions.
Squaring the above equation and using Jensen's inequality together with Lemma \ref{lem:heat_kernel} gives
\begin{align*}
     \Gamma(P_tf) = \abs*{\int_\sfG \nabla  p_t(\cdot,y) f(y) \dx{y}}^2 \leq C^2 e^{-2Kt} \int_\sfG p_t(\cdot,y) \abs*{\nabla f(y)}^2 \dx{y} = C^2 e^{-2Kt} P_t\Gamma(f).
\end{align*}
This proves the following. 

\begin{theorem}\label{thm:Bew_Graph}
    Let $\sfG=(\sfV,\sfE,\ell)$ be a metric graph. There exist constants $C\geq\deg_{\max}-1$ and $K>0$ such that for all $f\in W^{1,2}(\sfG)$ and $t\geq 0$ we have 
    \begin{align*}
         \Gamma(P_tf) \leq c^2(t) P_t\Gamma(f),
    \end{align*}
    where $c(t) = Ce^{-Kt}$. We can choose $K=\lambda_1$ for $\lambda_1>0$ the smallest non-zero eigenvalue of the Laplacian with standard boundary conditions. 
\end{theorem}

Having identified the Cheeger energy and its $L^2$-gradient flow, we now aim at connecting it to the $W_2$-flow of the entropy.
Let $\Ent:\Prob(\sfG)\to\R\cup\set{\infty}$, defined as 
\begin{align*}
    \Ent(\mu) = \begin{cases}
        \int_{\sfG} \eta(f) \dx{x} &: f = \frac{\dx{\mu}}{\dx{\leb}}\\ \infty &:\text{else}
    \end{cases} \quad \text{with} \quad \eta(r) = r \log r ,
\end{align*}
be the \emph{logarithmic entropy}. Note that, compared to \eqref{eq:Entr_refmeas}, we omit the explicit dependence on the reference measure and implicitly choose $\rmm = \leb$ to shorten the notation.
We denote by $\calI:\Prob(\sfG) \to [0,\infty]$ the \emph{Fisher Information}, given as the functional
\begin{align}\label{eq:FI}
    \calI(\mu) \coloneqq \begin{cases}\int_X \psi( f,\nabla f) \dx{x} &: \frac{\dx{\mu}}{\dx{\leb}} =  f \in W^{1,1}(\sfG)\cap C(\sfG)\\ \infty &:\text{else}
    \end{cases},
\end{align}
where the integrand is defined as
\begin{align*}
    \psi(u,v) = \begin{cases}
        \frac{\abs*{v}^2}{u} &: 0\neq 0\\ 0 &: u=0, \, v=0\\ \infty &: u=0, \, v\neq 0
    \end{cases}.
\end{align*}
As was shown in \cite{Erbar_2022}[Theorem 5.5], the heat equation \eqref{eq:heat} can be understood as the $W_2$-gradient flow of $\Ent$ in the following sense.

\begin{definition}\label{def:gf_entr}
    Let $T>0$. We say that $t\mapsto \mu_t \in \AC^2\bra*{[0,T];\Prob(\sfG)}$ is a \emph{$W_2$-gradient flow curve} of the logarithmic entropy if $\Ent(\mu_0)<\infty$ and
    \begin{align*}
        \calD(\mu_t) \coloneqq \Ent(\mu_T) - \Ent(\mu_0) + \frac{1}{2} \int_0^T \abs*{\Dot{\mu}_t}^2 + \calI(\mu_t) \dx{t} = 0.
    \end{align*}
\end{definition}

Since gradient flow curves of $\Ch$ are weak solutions of \eqref{eq:heat}, we conclude that both flows coincide when choosing $L^2(\sfG)$ initial data.

\begin{lemma}\label{lem:identify_gfs}
    Let $f\in L^2(\sfG)$ and $T>0$ be given. Then, $\mu_t = (P_tf)\leb$ defines a $W_2$-gradient flow curve of the logarithmic entropy. On the other hand, if $\mu_t = f_t\leb$ for $t\in[0,T]$ is a $W_2$-gradient flow curve of the entropy starting from $f\in L^2(\sfG)$, then $f_t= P_tf$ and it defines an $L^2(\sfG)$-gradient flow curve of the Cheeger energy.
\end{lemma}

\begin{proof}
    The curve $\mu_t = (P_tf)\leb$ solves the continuity equation weakly, see \cite{Erbar_2022}[Remark 2] and by \cite{Erbar_2022}[Theorem 5.5] it defines a gradient flow curve as in Definition \ref{def:gf_entr}. 
    Now, let $\mu_t = f_t\leb$ be a $W_2$-gradient flow curve of the logarithmic entropy starting from $f\in L^2(\sfG)$. From the first part of the proof we know that $(P_tf)\leb$ for $t\in[0,T]$ defines a gradient flow curve as well, implying that $\calD(\mu_t) = 0 = \calD\bra*{(P_tf)\leb}$. 
    Suppose that $f_t \neq P_tf$ and define $\Tilde{\mu}_t = \Tilde{f}_t\leb$ with $\Tilde{f_t} = \lambda f_t + (1-\lambda)P_tf$ for $\lambda\in(0,1)$. The functional $\calD$ from Definition \ref{def:gf_entr} is strictly convex due to the strict convexity of the logarithmic entropy, see for example \cite{heinze_gradient_2024}[Lemma 3.11]. Therefore, we obtain 
    \begin{align*}
        \calD(\Tilde{\mu}_t) < \lambda \calD(\mu_t) + (1-\lambda) \calD\bra*{(P_tf)\leb} = 0,
    \end{align*}
    which contradicts the non-negativity of the functional $\calD$ shown in \cite{Erbar_2022}[Theorem 5.5]. This gives $f_t= P_tf$.
\end{proof}

\begin{remark}
    Following the same arguments as in the previous proof and employing strict convexity of the functional $\calD$, gradient flow curves from Definition \ref{def:gf_entr} are unique for fixed initial data $\mu_0\in\Prob(\sfG)$ with $\Ent(\mu_0) < \infty$.
\end{remark}

Having established Lemma \ref{lem:identify_gfs}, we can unambiguously use the notation $(P_t)_{t\geq 0}$ and $(H_t)_{t\geq 0}$ for the heat semigroup and its dual associated to \eqref{eq:heat} representing weak solutions.

\section{Curvature bounds}\label{sec:CurvatureBnds}

In this section, we introduce the three equivalent notions of weak curvature bounds on metric graphs in the spirit of \cite{stefani_generalized_2022}[Theorem 5.16]. Some of our arguments rely on technical tools from the general theory of metric measure spaces. As a result of the identification from the previous section, we can verify the assumptions necessary for this abstract theory to apply, see \cite{stefani_generalized_2022}[Section 2]. 

\begin{lemma}\label{lem:MetrPropGraphs}
    The following hold true. 
    \begin{enumerate}
        \item[$(\rm A1)$] The space $(\sfG, d)$ is a complete and separable metric space. 
        \item[$(\rm A2)$] The Lebesgue measure $\leb\in \MeasPos(\sfG)$ is finite on bounded sets and supported on the whole space $\sfG$. 
        \item[$(\rm A3)$] For all $x\in\sfG$ and $r>0$ it holds that $\leb\bra*{B_r(x)} \leq \deg_{\max} e^{r^2}$.
        \item[$(\rm A4)$] We have that $\Ch(f+g) + \Ch(f-g) = 2\Ch(f) + 2\Ch(g)$ for all $f,g\in\Dom(\Ch)$.
        \item[$(\rm A5)$] If $f\in \Dom(\Ch)$ with $\weak{f}\leq L$ $\leb$-a.e. for some constant $L\geq 0$, then $f=g$ $\leb$-a.e. with $g\in {\rm Lip}(\sfG)$ and $\Lip{g}\leq L$.
    \end{enumerate}
    In addition, the graph distance coincides with
    \begin{align*}
        d_{\Ch}(x,y) \coloneqq \sup\set*{\abs*{f(x) - f(y)} \setsep f\in \Dom(\Ch)\cap C(\sfG), \, \Gamma(f)\leq 1 \; \leb\text{-a.e.}}.
    \end{align*}
\end{lemma}

\begin{proof}
    The first property has been discussed in \cite{mugnolo_semigroup_2014}[Lemma 3.17] and the second one is a direct consequence of the properties of the one-dimensional Lebesgue measure. 
    The third statement follows from the bounds $\leb\bra*{B_r(x)} \leq \deg_{\max} r$ and $e^{r^2} \geq r$ for $r\geq 0$. 
    Using the representation provided in Theorem \ref{thm:IdCheeger}, we have that $f\pm g \in \Dom(\Ch) $ for $f,g\in \Dom(\Ch)$ and
    \begin{align*}
        \Ch(f+g) + \Ch(f-g) &= \frac{1}{2}\int_\sfG \abs*{\nabla(f+g)}^2 \dx{\leb} + \frac{1}{2}\int_\sfG \abs*{\nabla(f-g)}^2 \dx{\leb}\\
        &= \int_\sfG \frac{1}{2}\abs*{\nabla f}^2 + \frac{1}{2}\abs*{\nabla g}^2 + \abs*{\nabla f}\abs*{\nabla g} + \frac{1}{2}\abs*{\nabla f}^2 + \frac{1}{2}\abs*{\nabla g}^2 - \abs*{\nabla f}\abs*{\nabla g} \dx{\leb} \\
        &= 2\Ch(f) + 2\Ch(g).
    \end{align*}
    From Theorem \ref{thm:IdCheeger} we know that $\Dom(\Ch) = W^{1,2}(\sfG)$ and, if $f\in \Dom(\Ch)$ such that $\weak{f} = \abs*{\nabla f} \leq L$ $\leb$-a.e., then $f\in W^{1,\infty}(\sfG)$ by definition. Theorem \ref{thm:Rademacher} and its proof imply $f=g$ $\leb$-a.e. with $g\in {\rm Lip}(\sfG)$ and $\Lip{g}\leq L$, which verifies the last property.
    The identification of the distances follows from \cite{baudoin_differential_2018} and Theorem \ref{thm:IdCheeger}.
\end{proof}

Another important property of the heat flow is the heat-smoothing effect. In view of Theorem \ref{thm:IdCheeger} we obtain the following result.

\begin{lemma}\label{lem:HeatSmooth}
    The heat semigroup on a metric graph is \emph{heat-smoothing}, meaning that for all $f\in L^\infty(\sfG)\cap W^{1,2}(\sfG)$ and $t>0$ it holds that $\weak{P_t f} = \lip{P_tf}$ $\leb$-a.e. for all $t>0$.
\end{lemma}

\begin{proof}
    As was shown in \cite{fijavz_variational_2007}[Theorem 3.6], $P_tf^\sfe \in C^\infty\bra*{\pra*{-\frac{\ell_\sfe}{2}, \frac{\ell_\sfe}{2}}}$ for each edge $\sfe\in\sfE$ and $P_tf \in C(\sfG)$ for all $t>0$. In particular, $P_tf\in W^{1,\infty}(\sfG)$ and by Theorem \ref{thm:Rademacher} we have that $P_tf \in {\rm Lip}(\sfG)$ with $\weak{P_tf} \leq \lip{P_tf}$ $\leb$-a.e. in $\sfG$.    
    Now, let $x\in \bra*{-\frac{\ell_\sfe}{2}, \frac{\ell_\sfe}{2}}$ be given with a sequence $(y_n)_{n\in\N}\subset\sfe$ such that $y_n\to x$ on some edge $\sfe\in\sfE$. For any $y_n\in\sfe$ there exists $z_n\in [x,y_n]$ with 
    \begin{align*}
        \frac{\abs*{P_tf^\sfe(x) - P_tf^\sfe(y_n)}}{\abs{x-y_n}} = \abs*{\nabla P_tf^\sfe(z_n)}
    \end{align*}
    by the mean value theorem. In particular, $z_n\to x$ as $y_n \to x$ and by smoothness of $P_tf$ we obtain 
    \begin{align*}
        \lim\limits_{n\to\infty}\frac{\abs*{P_tf^\sfe(x) - P_tf^\sfe(y_n)}}{\abs{x-y_n}} = \abs*{\nabla P_tf^\sfe(x)}.
    \end{align*}
    Since $y_n\to x$ was arbitrary and $\sfV\subset\sfG$ is of Lebesgue measure zero, this implies
    \begin{align*}
        \lip{P_t f}(x) \leq \abs*{\nabla P_tf(x)} = \weak{P_tf}(x) \quad \leb\text{-a.e.},
    \end{align*}
    where the last equality follows from Theorem \ref{thm:IdCheeger}.
\end{proof}

In order to formulate and prove our main result, we establish an equivalence between the weak gradient bound from Theorem \ref{thm:Bew_Graph} and a $W_2$-contraction estimate first. It is also called Kuwada duality and has been shown initially in \cite{kuwada_duality_2010} and stated for our setting in \cite{ambrosio_bakryemery_2015} with a detailed proof given in \cite{stefani_generalized_2022}. Thanks to Lemma  \ref{lem:MetrPropGraphs}, the result applies to metric graphs as well.

\begin{theorem}\label{thm:kuwada_dual}
    On a metric graph $\sfG=(\sfV,\sfE,\ell)$ let $c(t) = Ce^{-Kt}$ be as in Theorem \ref{thm:Bew_Graph}. Then, \eqref{eq:BEw} is equivalent to the existence of a dense subset $\calS\subset\ProbAc(\sfG)$ with $H_t(\calS)\subset\ProbAc(\sfG)$ such that for all $f\leb, g\leb \in \calS$ and $t\geq 0$ it holds that 
    \begin{align*}
        W_2\bra*{(P_tf)\leb, (P_tg)\leb} \leq c(t) W_2(f\leb, g\leb).
    \end{align*}
    Moreover, for all $t\geq0$ the dual heat semigroup $H_t$ extends uniquely to a map $H_t:\Prob(\sfG) \to \Prob(\sfG)$ such that for all $\mu,\nu\in\Prob(\sfG)$ we have
    \begin{align}\label{eq:contr_weak}\tag{${\rm K}_w$}
        W_2(H_t\mu, H_t\nu) \leq c(t) W_2(\mu,\nu).
    \end{align}
\end{theorem}

\begin{proof}
    The result follows as a special case of \cite{stefani_generalized_2022}[Theorem 3.16], applicable because of Lemma \ref{lem:MetrPropGraphs} and Theorem \ref{thm:Bew_Graph}.
\end{proof}

We are now in a position to state our main result.

\begin{theorem}\label{thm:equivalence}
    Let $\sfG=(\sfV,\sfE,\ell)$ be a metric graph and let 
    \begin{align*}
        R(t_0,t_1) = \int_0^1 c^{-2}\bra*{(1-s)t_0 + s t_1} \dx{s}
    \end{align*}
    for $c(t) = Ce^{-Kt}$ be as in Theorem \ref{thm:Bew_Graph}. Then, the following are equivalent.
    \begin{enumerate}[label=\roman*)]
        \item The \emph{weak Bakry-Émery estimate} holds, that is, for all $f\in W^{1,2}(\sfG)$ and $t\geq0$ we have that \label{be_equiv}  
        \begin{align}\label{eq:BEw}\tag{${\rm BE}_w$}
            \Gamma(P_tf)\leq c^2(t) P_t\Gamma(f).
        \end{align}
        \item The \emph{weak Evolutionary Variational Inequality}, $\EVIw$ for short, is satisfied. That is, if $\mu_0 \in \Dom(\Ent)$ and $\mu_1\in\Prob(\sfG)$, then for all $0\leq t_0\leq t_1$ with $\mu_1\in\Dom(\Ent)$ if $t_0=0=t_1$ it holds that \label{EVI_equiv}
        \begin{align}\label{eq:EVI_weak}\tag{${\rm EVI}_w$}
            \frac{1}{2}W_2^2(H_{t_1}\mu_1, H_{t_0}\mu_0) - \frac{1}{2R(t_0,t_1)} W_2^2(\mu_1,\mu_0) \leq (t_1-t_0)\bra*{\Ent(H_{t_0}\mu_0) - \Ent(H_{t_1}\mu_1)}.
        \end{align}
        \item The metric graph satisfies the \emph{weak Riemannian Curvature-Dimension Condition}, denoted by $\RCDw$, meaning that \eqref{eq:contr_weak} is satisfied and that for $s\mapsto \mu_s$ an arbitrary Wasserstein geodesic connecting $\mu_0,\mu_1\in\Dom(\Ent)$ we have \label{RCD_equiv}
        \begin{align}\label{eq:RCD_weak}\tag{${\rm RCD}_w$}
            \Ent(H_{t+h}\mu_s) \leq &(1-s)\Ent(H_t\mu_0) + s\Ent(H_t\mu_1)\\& + \frac{s(1-s)}{2h}\bra*{\frac{1}{R(t,t+h)} W_2^2(\mu_0,\mu_1) - W_2^2(H_t\mu_0, H_t\mu_1)} \notag
        \end{align}
        for all $s\in[0,1]$, $t\geq0$ and $h>0$.
    \end{enumerate}
\end{theorem}

\begin{remark}\label{rem:EVI_to_Kuwada}
    Using the identity
    \begin{align*}
        R(t_0,t_1) = \begin{cases}
            \frac{1}{2C^2K(t_1-t_0)}\bra*{e^{2Kt_1} - e^{2Kt_0}} &: t_0\neq t_1 \\
            C^{-2}e^{2Kt} &: t_0=t=t_1
        \end{cases},
    \end{align*}
    inequality \eqref{eq:EVI_weak} coincides with \eqref{eq:contr_weak} when choosing $t_0=t=t_1$.
\end{remark}

The remainder of this section is devoted to the proof of Theorem \ref{thm:equivalence}, closely following the arguments presented in \cite{stefani_generalized_2022}[Section 5] and the proof of the equivalence result therein, adapted to the new regularization on metric graphs.

\subsection{Strongly regular curves}

The main difficulty in the proof of Theorem \ref{thm:equivalence} is showing that \eqref{eq:BEw} implies \eqref{eq:EVI_weak}. In general, the curves we consider will lack the regularity necessary to show the estimate directly.
For this reason, we establish an inequality close to \eqref{eq:EVI_weak} for a regularized version of the curve $s\mapsto \mu_s$ first using the notion of strongly regular curves introduced in \cite{stefani_generalized_2022}[Definition 5.1]. 

\begin{definition}\label{def:StronglyReg}
    A curve $s\mapsto \mu_s\in\AC\bra*{[0,1];\Prob(X)}$ is said to be \emph{strongly regular} if $\frac{\dx{\mu_s}}{\dx{\leb}} = f_s$ for all $s\in[0,1]$ with 
    \begin{align*}
        s\mapsto f_s \in C^1\bra*{[0,1]; L^2(\sfG)}.
    \end{align*}
\end{definition}

As a consequence of Hölder's and Jensen's inequalities, the entropy along strongly regular curves on metric graphs is finite and bounded. Moreover, we can characterize the change in the action functional for such curves. 

\begin{lemma}\label{lem:DerivActReg}
    Let $s\mapsto \mu_s \in \AC^2\bra*{[0,1];\Prob(\sfG)}$ be a strongly regular curve with $f_s = \frac{\dx{\mu_s}}{\dx{\leb}}$ and fix $h\in C^1\bra*{[0,1];[0,1]}$, $\vartheta\in C^2\bra*{[0,1];[0,\infty)}$ such that $h(0) = 0$, $h(1)=1$ and $\Dot{\vartheta}\geq 0$, $\vartheta(s)>0$ for all $s>0$.
    We define 
    \begin{align*}
        s\mapsto \Tilde{\mu}_s = H_{\vartheta(s)}\mu_{h(s)} \in C\bra*{[0,1];\Prob(\sfG)}
    \end{align*}
    with density $\Tilde{f}_s = \frac{\dx{\Tilde{\mu}_s}}{\dx{\leb}}$. Let $\varphi\in {\rm Lip}_{+}(\sfG)$ be given and set
    \begin{align*}
        \varphi_s = Q_s\varphi(x) \coloneqq \inf\limits_{y\in \sfG} \varphi(y) + \frac{1}{2s}d(x,y)
    \end{align*}
    for $s\in[0,1]$ and $Q_0\varphi = \varphi$. Then, for a.e. $s\in(0,1)$ it holds that 
    \begin{align*}
        s\mapsto \int_\sfG \varphi_s \dx{\Tilde{\mu}_s} \in {\rm Lip}\bra*{[0,1]}
    \end{align*}
    and for the derivative we have that
    \begin{align*}
        \frac{\dx{}}{\dx{s}} \int_\sfG \varphi_s \dx{\Tilde{\mu}_s} &= - \frac{1}{2}\int_\sfG \lip{\varphi_s}^2 \dx{\Tilde{\mu}_s} - \Dot{\vartheta}(s) \int_\sfG \nabla f \cdot \nabla g \dx{\leb} + \Dot{h}(s)\int_\sfG \Dot{f}_{h(s)} P_{\vartheta(s)}\varphi_s \dx{\leb}.
    \end{align*}
\end{lemma}

\begin{proof}
    This is the content of \cite{stefani_generalized_2022}[Lemma 5.3] which is applicable because of Lemma \ref{lem:MetrPropGraphs}.
\end{proof}

\begin{remark}
    The definition $\varphi_s = Q_s\varphi$ for $s\geq 0$ induces a semigroup, namely the Hopf-Lax semigroup. It is closely related to the dual formulation of optimal transport and has already been studied in the context of gradient flows on metric graphs in \cite{Erbar_2022}[Section 3.4].
\end{remark}

Similar to the action estimate from the previous lemma, we want to characterize the change of the entropy next. However, the entropy may not be differentiable if the curve vanishes on a set of positive Lebesgue measure.
For this reason, we regularize the integrand as well. 
For $\delta>0$ and $\mu\in\Prob(\sfG)$ we define 
\begin{align*}
    \eta_\delta(r) \coloneqq \log(r+\delta)\quad\text{and} \quad \Ent_\delta(\mu) = \begin{cases}
        \int_\sfG \eta_\delta(f) \dx{\mu} &: f = \frac{\dx{\mu}}{\dx{\leb}}\\ \infty &:\text{else}
    \end{cases},
\end{align*}
which implies $\Ent_\delta(\mu) \geq \Ent(\mu)$ for all $\delta>0$. Moreover, if $\mu\in \ProbAc(\sfG)$ with density $f\in L^2(\sfG)$, then it holds that 
\begin{align*}
    \Ent_\delta(\mu) &= \int_\sfG \eta_\delta(f) \dx{\mu} = \int_\sfG \eta_\delta(f) - \log(\delta) \dx{\mu} + \log(\delta) \\&= \int_\sfG \log\bra*{\frac{f+\delta}{\delta}} \dx{\mu} + \log(\delta) \leq  \frac{1}{\delta} \int_\sfG f^2 \dx{\leb} + \log(\delta) <\infty.
\end{align*}
In addition to this inequality, if $\mu\in\ProbAc(\sfG)$ with density $f\in L^2(\sfG)$ we have 
\begin{align}\label{eq:LimRegEntr}
    \Ent(\mu) = \lim\limits_{\delta\to 0} \Ent_\delta(\mu).
\end{align}
Indeed, for such $\mu\in\ProbAc(\sfG)$ and $\delta\in(0,1]$ we can estimate 
\begin{align*}
    \pra{\log(f+\delta)}^- \leq \pra{\log(\delta)}^- ~,\quad &\int_\sfG \pra{\log(f+\delta)}^-\dx{\mu} \leq \int_\sfG \pra{\log(f)}^-\dx{\mu} <\infty, 
\end{align*}
since an application of Hölder's inequality together with $-r\log(r) \leq \sqrt{r}$ for $r\in[0,1]$ gives
\begin{align*}
    \int_\sfG \pra{\log(f)}^-\dx{\mu} \leq \int_\sfG \sqrt{f} \dx{\leb} \leq C \sqrt{\norm{f}_{L^2(\sfG)}}
\end{align*}
for a constant $C>0$ depending only on $\leb(\sfG)$.
To see that $-r\log(r) \leq \sqrt{r}$ for $r\in[0,1]$, note that we can equivalently write $\sqrt{r}\bra*{1+\sqrt{r}\log(r)}\geq 0$ and the left-hand side is a product of non-negative terms for $r\in[0,1]$. 
For the positive part of the integrand it holds that
\begin{align*}
    \pra{\log(f+\delta)}^+ \leq \pra{\log(f+1)}^+ \leq f\quad \text{and}\quad &\int_\sfG \pra{\log(f+\delta)}^+\dx{\mu} \leq \int_\sfG f^2\dx{\leb} <\infty.
\end{align*}
This implies the existence of an integrable majorant, allowing us to apply the dominated convergence theorem and conclude \eqref{eq:LimRegEntr}.
To shorten the notation, we introduce the functions
\begin{align*}
    \hat{\eta}_\delta(r) &\coloneqq \eta_\delta(r) - \log(\delta),\quad
    p_\delta(r) \coloneqq \hat{\eta}_\delta(r) + r \hat{\eta}'_\delta(r)
\end{align*}
for $r\geq 0$. 
With this notation we have the following characterization of the time derivative of the regularized entropy along curves as in Lemma \ref{lem:DerivActReg}.  

\begin{lemma}\label{lem:DerivEntrReg}
    Let $s\mapsto \Tilde{\mu}_s \in C([0,1];\Prob(\sfG))$ be as in Lemma \ref{lem:DerivActReg} and let $\delta>0$. It holds that $s\mapsto \Ent_\delta(\Tilde{\mu}_s)\in C^1\bra*{(0,1]}$ with 
    \begin{align*}
        \frac{\dx{}}{\dx{s}} \Ent_\delta(\Tilde{\mu}_s) \leq - \Dot{\vartheta}(s) \int_\sfG \Gamma(g_s^\delta) \dx{\Tilde{\mu}_s} + \Dot{h}(s) \int_\sfG \Dot{f}_{h(s)} P_{\vartheta(s)} g_s^\delta \dx{\leb}
    \end{align*}
    for $s\in(0,1]$ and $g_s^\delta = p_\delta(\Tilde{f}_s)$.
\end{lemma}

\begin{proof}
    The statement has been shown in \cite{stefani_generalized_2022}[Lemma 5.5] in an abstract setting. We can apply the result because of Lemma \ref{lem:MetrPropGraphs}.
\end{proof}

Now, combining Lemma \ref{lem:DerivActReg} and Lemma \ref{lem:DerivEntrReg} results in an inequality that relates action and entropy along strongly regular curves. 

\begin{theorem}\label{thm:IneqReg}
    Let $s\mapsto \Tilde{\mu}_s \in C([0,1];\Prob(\sfG))$ be as in Lemma \ref{lem:DerivActReg}. For all $\delta>0$ it holds that 
    \begin{align*}
        \frac{1}{2} W_2^2\bra*{\Tilde{\mu}_1, \Tilde{\mu}_0} - \int_0^1 \Ddot{\vartheta}(s) \Ent_\delta(\Tilde{\mu}_s) \dx{s} + \Dot{\vartheta}(1) \Ent_\delta(\Tilde{\mu}_1) \leq \Dot{\vartheta}(0) \Ent_\delta(\Tilde{\mu}_0) + \frac{1}{2 I_\vartheta(1)} \int_0^1 \abs*{\Dot{\mu}_s}^2 \dx{s}, 
    \end{align*}
    where 
    \begin{align*}
        I_\vartheta(s) = \int_0^s c^{-2}\bra*{\vartheta(r)}\dx{r} \quad \text{and}\quad h(s) = \frac{I_\vartheta(s)}{I_\vartheta(1)}~, \quad s\in[0,1],
    \end{align*}
    with $c(t) = C e^{-Kt}$ for $t\geq 0$ from Theorem \ref{thm:Bew_Graph}.
\end{theorem}

\begin{proof}
    This is the content of \cite{stefani_generalized_2022}[Lemma 5.6]. Note that the result is applicable because of Lemma \ref{lem:MetrPropGraphs} and Lemma \ref{lem:HeatSmooth}.
\end{proof}

The choice $\vartheta(s) = (1-s)t_0 + s t_1$ in Theorem \ref{thm:IneqReg} recovers \eqref{eq:EVI_weak}, thus \eqref{eq:BEw} already implies \eqref{eq:EVI_weak} for strongly regular curves.

\subsection{Regularization of absolutely continuous curves}

In view of Theorem \ref{thm:IneqReg}, a crucial step in the proof of the equivalences is to show the existence of a sequence of strongly regular curves approximating any $\AC^2\bra*{[0,1];\Prob(\sfG)}$ curve. Our strategy is to regularize in two ways. First, we ensure the existence of densities with respect to the Lebesgue measure. To this end, we make use of the regularization procedure introduced in Definition \ref{def:Regular}. Next, we smooth in the time variable as in \cite{stefani_generalized_2022}[Theorem 5.15] by mollification, allowing us to obtain continuous differentiability. 

\begin{theorem}\label{thm:ExStrongReg}
    For any $s\mapsto \mu_s \in \AC^2\bra*{[0,1];\Prob(\sfG)}$ there exists a family of strongly regular curves $s\mapsto \mu^n_s \in \AC^2\bra*{[0,1];\Prob(\sfG^{2/n})}$ for $n\in\N$ such that 
    \begin{enumerate}[label=\roman*)]
        \item $\mu_s^n \weakcvg \mu_s$ in $\sfG$ for all $s\in[0,1]$ as $n\to\infty$,
        \item the action can be estimated by $\limsup_{n\to\infty} \int_0^1 \abs*{\Dot{\mu}^n_s} \dx{s} \leq  \int_0^1 \abs*{\Dot{\mu}_s} \dx{s}$ and
        \item for all $s\in[0,1]$, $t\geq 0$ we have convergence of the entropies $\lim\limits_{n\to\infty} \Ent^{1/n}(H_t\mu_s^n) = \Ent(H_t\mu_s)$, where $\Ent^{1/n}:\sfG^{2/n}\to[0,\infty]$ is the logarithmic entropy on the extended metric graph $\sfG^{2/n}$.
    \end{enumerate}
\end{theorem}

\begin{proof}
    The proof is divided into four steps. First, we extend the curve in time. Next, we smooth with respect to the space-dependence using the regularization from Definition \ref{def:Regular}. Then, we regularize in time by standard mollification and lastly, we restrict the time-dependence to the initial domain $[0,1]$ again.

    \textbf{Step 1: Time-extension.} 
    By slight abuse of notation we define the curve $\mu_s: \R \to \Prob(\sfG)$ as the constant extension 
    \begin{align*}
        \mu_s = \begin{cases}
            \mu_0 &: s\le 0\\ \mu_s &:0\leq s \leq 1\\ \mu_1 &: 1\le s
        \end{cases}
    \end{align*}
    keeping the absolute continuity and we can write $s\mapsto\mu_s \in \AC^2\bra*{\R; \Prob(\sfG)}$.

    \textbf{Step 2: Regularization in space.} 
    For $s\in\R$ and $\eps\in(0,1]$ we define $\mu_s^{\eps} \in \Prob(\sfG^{2\eps})$ by duality as in Definition \ref{def:Regular}. 
    From Proposition \ref{prop:PropertiesReg} we know that $\mu^\eps_s \weakcvg \mu_s$ for all $s\in\R$ as $\eps \to 0$ and $s\mapsto\mu_s^\eps \in \AC^2\bra*{\R; \Prob(\sfG^{2\eps})}$ with $\limsup_{\eps\to 0}\abs*{\Dot{\mu}^\eps_s} \leq \abs*{\Dot{\mu}_s}$ for a.e. $s\in\R$. Additionally, the same proposition gives $\mu^\eps_s \ll \leb$ with density denoted by $ f^\eps_s\in L^1(\sfG^{2\eps})$. 
    We extend all measures and their respective densities by zero to measures (densities) on the extended graph $\sfG^{2}$, noting that $\sfG\subset\sfG^{2\eps}\subset\sfG^2$ for all $\eps \in (0,1]$. Let us denote by $\Ent^\eps: \Prob(\sfG^{2\eps}) \to [0,\infty]$ the entropy defined on $\sfG^{2\eps}$ depending on $\eps\in(0,1]$ as well.  
    Since $\eta(0) = 0\cdot \log 0  = 0$ by continuous extension, we obtain
    \begin{align*}
        \Ent(\mu_s) = \Ent^1(\mu_s) \quad\text{and}\quad \Ent^{\eps}(\mu^\eps_s) = \Ent^1(\mu_s^\eps).
    \end{align*}
    Together with the weak lower semicontinuity of the logarithmic entropy, we infer that
    \begin{align}\label{eq:LiminfEntrAppr}
        \Ent(\mu_s) = \Ent^1(\mu_s) \leq \liminf\limits_{\eps\to0} \Ent^1(\mu_s^\eps)   = \liminf\limits_{\eps\to0} \Ent^\eps(\mu_s^\eps)  
    \end{align}
    for all $s\in\R$.
    Moreover, the equality $\Ent(\mu_s) = \lim_{\eps \to 0} \Ent^\eps(\mu_s^\eps)$ holds true. 
    If $\Ent(\mu_s) = \infty$, then $\Ent(\mu_s)\geq \limsup_{\eps\to0}\Ent^\eps(\mu_s^\eps)$, thus proving the claim. On the other hand, if $\Ent(\mu_s) <\infty$, then $\mu_s\ll\leb$ with $\frac{\dx{\mu_s}}{\dx{\leb}} =  f_s$ for $s\in\R$ and, using Proposition \ref{prop:PropertiesReg}, we can write
    \begin{align}
        \Ent^\eps(\mu_s^\eps) &= \int_\sfG \eta( f^\eps_s) \dx{\leb} + \int_{\sfG^{2\eps}\setminus\sfG} \eta( f^\eps_s) \dx{\leb} \notag\\
        &= \int_\sfG \eta\bra*{\frac{1}{2\eps}\int_{\sfe\cap I^\eps_\sfe(x)}  f_s(y) \dx{y}} \dx{x} + \int_{\sfG^{2\eps}\setminus\sfG} \eta\bra*{\frac{1}{2\eps} \sum\limits_{\sfe\in\sfE(\sfv)}\int_{\sfe\cap I^\eps_\sfe(x)}  f_s(y) \dx{y}} \dx{x} \label{eq:EntrEst_extended}
    \end{align}
    for all $\eps\in(0,1]$ and arbitrary $s\in\R$.
    In the above equation and in the following, we choose $\sfv\in\sfV$ in such a way that $x\in \sfe^{2\eps}_\sfv$ for $x\in\sfG^{2\eps}\setminus\sfG$. By construction, this choice is unique.
    We treat both terms in \eqref{eq:EntrEst_extended} separately. First, recall that $\eta:\R\to[0,\infty)$ is convex. In particular, the inequality
    \begin{align}\label{eq:cvx}
        \eta(\lambda r) \leq \lambda \eta(r) + (1-\lambda)\eta(0) = \lambda \eta(r)
    \end{align}
    holds true for all $\lambda\in[0,1]$ and $r\in\R$. Since $\leb(\sfe\cap I^{\eps}_\sfe(x)) \leq 2\eps$ for all $x\in\sfG$ and $\sfe\in\sfE$ such that $x\in\sfe$, we can choose $\lambda = \frac{\leb(\sfe\cap I^{\eps}_\sfe(x))}{2\eps}$ in this estimate.
    Applying Jensen's inequality in the second line and \eqref{eq:cvx} in the third, we conclude for the first term in \eqref{eq:EntrEst_extended} that 
    \begin{align*}
        \int_\sfG \eta\bra*{\frac{1}{2\eps}\int_{\sfe\cap I^\eps_\sfe(x)}  f_s(y) \dx{y}} \dx{x} &= \int_\sfG \eta\bra*{\frac{1}{\leb(\sfe\cap I^{\eps}_\sfe(x))}\int_{\sfe\cap I^\eps_\sfe(x)} \frac{\leb(\sfe\cap I^{\eps}_\sfe(x))}{2\eps} f_s(y) \dx{y}} \dx{x}\\
        &\leq \int_\sfG\frac{1}{\leb(\sfe\cap I^{\eps}_\sfe(x))} \int_{\sfe\cap I^\eps_\sfe(x)} \eta\bra*{\frac{\leb(\sfe\cap I^{\eps}_\sfe(x))}{2\eps} f_s(y)} \dx{y} \dx{x}\\
        &\leq \int_\sfG\frac{1}{2\eps}\int_{\sfe\cap I^\eps_\sfe(x)} \eta\bra*{ f_s(y)} \dx{y} \dx{x} \\&\eqqcolon \int_\sfG (\eta\circ f)^\eps \dx{x} .
    \end{align*}
    In a similar way, applying Jensen's inequality twice, the second summand can be bounded by
    \begin{align*}
        &\int_{\sfG^{2\eps}\setminus\sfG} \eta\bra*{\frac{1}{2\eps} \sum\limits_{\sfe\in\sfE(\sfv)}\int_{\sfe\cap I^\eps_\sfe(x)}  f_s(y) \dx{y}} \dx{x} \\
        =& \int_{\sfG^{2\eps}\setminus\sfG} \eta\bra*{ \frac{1}{\abs*{\sfE(\sfv)}}\sum\limits_{\sfe\in\sfE(\sfv)} \frac{1}{\leb(\sfe\cap I^{\eps}_\sfe(x))}\int_{\sfe\cap I^\eps_\sfe(x)} \frac{\leb(\sfe\cap I^{\eps}_\sfe(x))}{2\eps}\abs*{\sfE(\sfv)}  f_s(y) \dx{y}} \dx{x}\\
        \leq& \int_{\sfG^{2\eps}\setminus\sfG} \frac{1}{\abs*{\sfE(\sfv)}} \sum\limits_{\sfe\in\sfE(\sfv)} \frac{1}{\leb(\sfe\cap I^{\eps}_\sfe(x))} \int_{\sfe\cap I^\eps_\sfe(x)} \eta\bra*{\frac{\leb(\sfe\cap I^{\eps}_\sfe(x))}{2\eps}\abs*{\sfE(\sfv)}  f_s(y)} \dx{y} \dx{x}\\
        \leq& \int_{\sfG^{2\eps}\setminus\sfG} \frac{1}{\abs*{\sfE(\sfv)}} \sum\limits_{\sfe\in\sfE(\sfv)} \frac{1}{2\eps} \int_{\sfe\cap I^\eps_\sfe(x)} \eta\bra*{\abs*{\sfE(\sfv)}  f_s(y)} \dx{y} \dx{x} .
    \end{align*}
    For arbitrary $C>0$ and $r\in\R$ it also holds that 
    \begin{align*}
        \eta(Cr) = Cr\log(Cr)  = rC\log C + Cr\log r = C\eta(r) + r\eta(C).
    \end{align*}
    Applying the above inequality with the choice $C=\abs*{\sfE(\sfv)}$ then yields
    \begin{align*}
        &\int_{\sfG^{2\eps}\setminus\sfG} \frac{1}{\abs*{\sfE(\sfv)}} \sum\limits_{\sfe\in\sfE(\sfv)} \frac{1}{2\eps} \int_{\sfe\cap I^\eps_\sfe(x)} \eta\bra*{\abs*{\sfE(\sfv)}  f_s(y)} \dx{y} \dx{x}\\
        =& \int_{\sfG^{2\eps}\setminus\sfG} \sum\limits_{\sfe\in\sfE(\sfv)} \frac{1}{2\eps} \int_{\sfe\cap I^\eps_\sfe(x)} \eta\bra*{ f_s(y)} \dx{y} \dx{x} + \int_{\sfG^{2\eps}\setminus\sfG} \sum\limits_{\sfe\in\sfE(\sfv)} \log\bra*{\abs*{\sfE(\sfv)}}\frac{1}{2\eps} \int_{\sfe\cap I^\eps_\sfe(x)}  f_s(y) \dx{y} \dx{x}\\
        \leq& \int_{\sfG^{2\eps}\setminus\sfG} \bra*{\eta\circ f_s}^\eps(x) \dx{x} + \log\bra*{\deg_{\max}}\int_{\sfG^{2\eps}\setminus\sfG}  f_s^\eps(x) \dx{x},
    \end{align*}
    where we use the notation
    \begin{align*}
        \bra*{\eta\circ f_s}^\eps(x) \coloneqq \sum\limits_{\sfe\in\sfE(\sfv)} \frac{1}{2\eps} \int_{\sfe\cap I^\eps_\sfe(x)} \eta\bra*{ f_s(y)} \dx{y}
    \end{align*}
    for $x\in \sfG^{2\eps}\setminus\sfG$. From these two estimates we obtain
    \begin{align}
        \limsup\limits_{\eps\to 0}\Ent^\eps(\mu_s^\eps)
        \leq \limsup\limits_{\eps\to 0} &\int_{\sfG^{2\eps}} \bra*{\eta\circ f_s}^\eps(x) \dx{x} + \log\bra*{\deg_{\max}} \limsup\limits_{\eps\to 0}\mu^\eps\bra*{\sfG^{2\eps}\setminus\sfG}.\label{eq:split_limsup}
    \end{align}
    Since $\Ent(\mu_s)<\infty$ by assumption, we have $\eta\circ f_s \in L^1(\sfG)$ and we can define a measure $\Tilde{\mu}_s\in \Meas(\sfG^{2\eps})$ from this density by $\frac{\dx{\Tilde{\mu}_s}}{\dx{\leb}} = \eta\circ f_s$. Then, constructing $\Tilde{\mu}_s^\eps$ as in Definition \ref{def:Regular}, we obtain that $(\eta\circ f)^\eps$ coincides with $\Tilde{f}^\eps_s = \frac{\dx{\Tilde{\mu}_s^\eps}}{\dx{\leb}}$ from Proposition \ref{prop:PropertiesReg} \ref{density}. The weak convergence now gives
    \begin{align*}
        \limsup\limits_{\eps\to 0} \int_{\sfG^{2\eps}} \bra*{\eta\circ f_s}^\eps(x) \dx{x} = \limsup\limits_{\eps\to 0} \int_{\sfG^{2\eps}}  \dx{\Tilde{\mu}_s^\eps} = \int_\sfG \dx{\Tilde{\mu}_s} = \int_\sfG (\eta\circ f_s)(x) \dx{x} = \Ent(\mu_s).
    \end{align*}
    For the second term in \eqref{eq:split_limsup} we extend $\mu_s^\eps$ and $\mu_s$ by zero to measures on $\sfG^2$ and because of $\mu_s\ll\leb{}$ it holds that $\mu_s^\eps(\sfV) = 0 = \mu_s(\sfG^2\setminus\sfG \cup \sfV)$. 
    Again, from the weak convergence established in Proposition \ref{prop:PropertiesReg} we infer that
    \begin{align*}
        \int_{\sfG^2} \varphi(x) \dx{\mu_s^\eps(x)} &= \int_{\sfG^{2\eps}} \varphi\restrto{\sfG^{2\eps}}(x) \dx{\mu_s(x)} \longrightarrow \int_\sfG \varphi\restrto{\sfG}(x) \dx{\mu_s(x)} = \int_{\sfG^2} \varphi(x) \dx{\mu_s(x)},
    \end{align*}
    holds for all $\varphi\in C(\sfG^2)$, which gives $\mu_s^\eps\weakcvg\mu_s$ on $\sfG^2$. Applying the generalized version of Fatou's lemma \cite{ambrosio_bakryemery_2015}[Lemma 3.3] to the constant sequence $ \one_{(\sfG^2\setminus\sfG)\cup\sfV}$ then yields
    \begin{align*}
        \limsup\limits_{\eps\to 0} \mu_s^\eps(\sfG^{2\eps}\setminus\sfG) = \limsup\limits_{\eps\to 0} \mu_s^\eps\bra*{(\sfG^2\setminus\sfG) \cup\sfV} \leq \mu\bra*{(\sfG^2\setminus\sfG)\cup\sfV} = 0.
    \end{align*}
    Combining these estimates, \eqref{eq:split_limsup} can be bounded from above by
    \begin{align*}
        \limsup\limits_{\eps\to 0}\Ent^\eps(\mu_s^\eps) \leq \Ent(\mu_s) + 0 = \Ent(\mu_s)
    \end{align*}
    and thus, together with \eqref{eq:LiminfEntrAppr}, we obtain the desired convergence 
    \begin{align*}
        \lim\limits_{\eps\to 0}\Ent^\eps(\mu_s^\eps) = \Ent(\mu_s)
    \end{align*}
    for arbitrary $s\in\R$.
    
    \textbf{Step 3: Regularization in time.}
    Fix $\eps\in(0,1]$ and let $\zeta:\R \to \R$ be a standard mollifier, meaning that $\zeta\in C^\infty_c(\R)$ with support contained in $[-1,1]$, $0\leq\zeta\leq 1$ and $\int_\R\zeta(\tau)\dx{\tau}=1$. Recall that $\mu_s^\eps = f_s^\eps \leb$ with $f_s^\eps\in L^2(\sfG)$ as in Proposition \ref{prop:PropertiesReg}.
    For $k\in\N$ we define $\zeta_k(\tau) = k\zeta(k\tau)$ and for arbitrary $s\in\R$ let
    \begin{align*}
         f^{\eps,k}_s \coloneqq \bra*{\zeta_k \ast  f_\cdot^\eps}(s) = \int_\R \zeta_k(s-\tau)  f^\eps_\tau \dx{\tau}  = \int_\R \zeta_k(\tau)  f^\eps_{s-\tau} \dx{\tau} 
    \end{align*}
    which induces the measure $\mu_s^{\eps,k} \in \Prob(\sfG^{2\eps})$ by 
    \begin{align*}
        \mu_s^{\eps,k} (A) = \int_A  f_s^{\eps,k}(x) \dx{\leb(x)},
    \end{align*}
    or equivalently $ f^{\eps,k}_s = \frac{\dx{\mu^{\eps,k}_s}}{\dx{\leb}}$. Note that this regularization procedure can be applied to the curve $s\mapsto\mu_s$ as well with density $\frac{\dx{\mu_s}}{\dx{\leb}} = f_s$, leading to the regularized measure $\mu_s^k$ instead. In particular $s\mapsto \mu_s^k \in \AC^2(\R;\Prob(\sfG))$, allowing us to conclude from the previous step that
    \begin{align}\label{eq:CvgFixk}
        \lim\limits_{\eps\to 0} \Ent^\eps(\mu_s^{\eps,k}) = \Ent(\mu_s^k).
    \end{align}
    It remains to characterize the limit as $k\to \infty$. First, we show weak convergence of the regularized measure. 
    Let $s,s'\in\R$ be given and let $\pi_{s,s'}^\eps\in \Pi(\mu_s^\eps, \mu_{s'}^\eps)$ be an admissible transport plan as in \eqref{eq:AdmPlans}. We define $\pi^{\eps,k}_s\in\Prob(\sfG^{2\eps}\times\sfG^{2\eps})$ by duality as the measure on $\sfG^{2\eps}\times\sfG^{2\eps}$ such that for all continuous functions $\varphi:\sfG^{2\eps}\times\sfG^{2\eps}\to\R$ it holds that 
    \begin{align*}
        \int_{\sfG^{2\eps}\times\sfG^{2\eps}} \varphi(x,y)\dx{\pi^{\eps,k}_s(x,y)} = \int_\R \zeta_k(s-\tau) \int_{\sfG^{2\eps}\times\sfG^{2\eps}} \varphi(x,y) \dx{\pi^{\eps}_{s,\tau}(x,y)} \dx{\tau}.
    \end{align*}
    Note that the right-hand side indeed defines a continuous linear functional on $C(\sfG^{2\eps}\times\sfG^{2\eps})$ since 
    \begin{align*}
        &\int_\R \zeta_k(s-\tau) \int_{\sfG^{2\eps}\times\sfG^{2\eps}} \varphi(x,y) \dx{\pi^{\eps}_{s,\tau}(x,y)} \dx{\tau} \\&\leq \norm{\varphi}_{C(\sfG^{2\eps}\times\sfG^{2\eps})} \int_\R \zeta_k(s-\tau) \pi^{\eps}_{s,\tau}(\sfG^{2\eps}\times\sfG^{2\eps}) \dx{\tau} 
        = \norm{\varphi}_{C(\sfG^{2\eps}\times\sfG^{2\eps})} \int_\R \zeta_k(s-\tau) \dx{\tau} =  \norm{\varphi}_{C(\sfG^{2\eps}\times\sfG^{2\eps})}.
    \end{align*}
    Moreover, it holds that 
    \begin{align*}
        \int_{\sfG^{2\eps}\times\sfG^{2\eps}} \varphi(x) \dx{\pi^{\eps,k}_s(x,y)} = \int_\R \zeta_k(s-\tau) \int_{\sfG^{2\eps}} \varphi(x) \dx{\mu^\eps_s(x)} \dx{\tau} = \int_{\sfG^{2\eps}} \varphi(x) \dx{\mu^\eps_s(x)}
    \end{align*}
    for all $\varphi\in C(\sfG^{2\eps})$ and
    \begin{align*}
        \int_{\sfG^{2\eps}\times\sfG^{2\eps}} \varphi(y) \dx{\pi^{\eps,k}_s(x,y)} = \int_\R \zeta_k(s-\tau) \int_{\sfG^{2\eps}} \varphi(y) \dx{\mu^\eps_\tau(y)} \dx{\tau} = \int_{\sfG^{2\eps}} \varphi(y) \dx{\mu^{\eps,k}_s(y)}
    \end{align*}
    respectively, so that $\pi^{\eps,k}_s\in \Pi(\mu^\eps_s, \mu^{\eps,k}_s)$.
    Choosing $\pi^\eps_{s,\tau}\in\Pi(\mu^\eps_s,\mu^\eps_\tau)$ optimal, thus minimizing \eqref{eq:DefW2}, we obtain
    \begin{align*}
        W_2^2(\mu^{\eps,k}_s, \mu^{\eps}_s) &\leq \int_{\sfG^{2\eps}\times\sfG^{2\eps}} d^2(x,y) \dx{\pi^{\eps,k}_s(x,y)} = \int_\R \zeta_k(s-\tau) \int_{\sfG^{2\eps}\times\sfG^{2\eps}} d^2(x,y) \dx{\pi^\eps_{s,\tau}(x,y)} \dx{\tau} \\
        &= \int_\R \zeta_k(s-\tau) W_2^2(\mu^\eps_s, \mu^\eps_\tau) \dx{\tau} = (\zeta_k\ast W_2^2(\mu^\eps_s, \mu^\eps_\cdot))(s),
    \end{align*}
    where we used that $d^2(\cdot,\cdot)\in C(\sfG^{2\eps}\times\sfG^{2\eps})$. The limit as $k\to\infty$ reads
    \begin{align*}
        0 \leq\lim\limits_{k\to\infty} W_2^2(\mu^{\eps,k}_s, \mu^{\eps}_s) \leq \lim\limits_{k\to\infty} (\zeta_k\ast W_2^2(\mu^\eps_s, \mu^\eps_\cdot))(s) = 0,
    \end{align*}
    which implies $\mu^{\eps,k}_s\weakcvg \mu^\eps_s$ as $k\to\infty$ for all $s\in\R$. Again, the weak lower semicontinuity of the logarithmic entropy gives
    \begin{align*}
        \Ent^\eps(\mu_s^\eps) \leq \liminf\limits_{k\to\infty} \Ent^\eps(\mu^{\eps,k}_s)
    \end{align*}
    and we can show equality as well. For any $s,s'\in\R$ let $\pi^\eps_{s,s'}\in \Pi(\mu^\eps_s, \mu^\eps_{s'})$ be optimal. We define $\Bar{\pi}^{\eps,k}_{s,s'}\in \Prob(\sfG^{2\eps}\times\sfG^{2\eps})$ for $n\in\N$ by duality such that for all $\varphi\in C(\sfG^{2\eps}\times\sfG^{2\eps})$ we have
    \begin{align*}
        \int_{\sfG^{2\eps}\times\sfG^{2\eps}} \varphi(x,y)\dx{\Bar{\pi}^{\eps,k}_{s,s'}(x,y)} = \int_\R \zeta_k(\tau) \int_{\sfG^{2\eps}\times\sfG^{2\eps}} \varphi(x,y) \dx{\pi^{\eps}_{s-\tau, s'-\tau}(x,y)} \dx{\tau}.
    \end{align*}
    Following similar arguments as before, $\Bar{\pi}^{\eps,k}_{s,s'} \in \Pi(\mu^{\eps,k}_s, \mu^{\eps,k}_{s'})$ and we obtain 
    \begin{align*}
        W_2^2(\mu^{\eps,k}_s, \mu^{\eps,k}_{s'}) &\leq \int_{\sfG^{2\eps}\times\sfG^{2\eps}} d^2(x,y) \dx{\Bar{\pi}^{\eps,k}_{s,s'}(x,y)} = \int_\R \zeta_k(\tau) W_2^2(\mu^{\eps}_{s-\tau}, \mu^\eps_{s'-\tau}) \dx{\tau}\\
        &\leq \int_\R \zeta_k(\tau)\int_{s-\tau}^{s'-\tau} \abs{\Dot{\mu}^\eps_r}^2 \dx{r}\dx{\tau} = \int_s^{s'} \bra*{\zeta_k \ast \abs{\Dot{\mu}^\eps_\cdot}^2}(r) \dx{r}.
    \end{align*}
    Therefore, $s\mapsto \mu^{\eps,k}_s \in \AC^2(\R; \Prob(\sfG^{2\eps})$ with
    \begin{align*}
        \abs{\Dot{\mu}^{\eps,k}_s}^2 \leq \bra*{\zeta_k \ast \abs{\Dot{\mu}^\eps_\cdot}^2}(s)
    \end{align*}
    for a.e. $s\in\R$. Additionally, Jensen's inequality gives 
    \begin{align*}
        \Ent^\eps(\mu^{\eps,k}_s) &= \int_{\sfG^{2\eps}} \eta\bra*{ f^{\eps,k}_s(x)}\dx{x} \leq \int_{\sfG^{2\eps}} \int_\R \zeta_k(s-\tau) \eta\bra*{ f^\eps_\tau (x)}  \dx{\tau}{\dx{x}} = \bra*{\zeta_k\ast\Ent^\eps( \mu^\eps_\cdot)}(s).
    \end{align*}
    Taking the limit as $k\to\infty$ then yields
    \begin{align*}
        \lim\limits_{k\to\infty} \Ent^\eps( \mu^{\eps,k}_s) \leq \Ent^\eps( \mu^\eps_s)
    \end{align*}
    for all $s\in\R$, thus equality.

    \textbf{Step 4: Conclusion.}
    We restrict the curve to the time-interval $[0,1]$ on which it is absolutely continuous by construction, keeping the notation $s \mapsto \mu_s^{\eps,k} \in \AC^2\bra*{[0,1];\Prob(\sfG^{2\eps})}$.
    From the previous steps it follows that
    \begin{align*}
        \lim\limits_{\eps, k} \Ent^\eps(\mu_s^{\eps,k}) = \lim\limits_{k\to\infty}\lim\limits_{\eps\to0} \Ent^\eps(\mu_s^{\eps,k}) = \lim\limits_{\eps\to0}\lim\limits_{k\to\infty} \Ent^\eps(\mu_s^{\eps,k})
    \end{align*}
    for all $s\in[0,1]$. 
    Let $\calR^{\eps, k}: \AC^2\bra*{[0,1];\Prob(\sfG)} \to \AC^2\bra*{[0,1];\Prob(\sfG^{2\eps})}$, defined by $(s\mapsto \mu_s) \mapsto (s\mapsto \mu_s^{\eps,k})$ denote the regularization map. This map is linear by construction and therefore commutes with the dual heat flow
    \begin{align*}
        H_t \calR^{\eps,k} = \calR^{\eps,k} H_t
    \end{align*}
    for all $t\geq 0$. This gives
    \begin{align*}
        \lim\limits_{\eps, k} \Ent^\eps(H_t\mu_s^{\eps,k}) = \lim\limits_{\eps, k} \Ent^\eps\bra*{H_t(R^{\eps,k}\mu_s)} = \lim\limits_{\eps, k} \Ent^\eps\bra*{R^{\eps,k} H_t\mu_s} = \Ent(H_t\mu_s). 
    \end{align*}
    The choice $k=n$, $\eps = \frac{1}{n}$ and a diagonalization argument conclude the proof.
\end{proof}

With these preparations, we are now in a position to show the equivalences.

\begin{proof}[Proof of Theorem \ref{thm:equivalence}]
    We first show the relation between the weak Bakry-Émery estimate \eqref{eq:BEw} and the weak Evolutionary Variational Inequality \eqref{eq:EVI_weak}. Next, we establish the equivalence between \eqref{eq:BEw} and the weak curvature condition \eqref{eq:RCD_weak}.

    \textbf{\eqref{eq:BEw}$\Rightarrow$\eqref{eq:EVI_weak}:}
    By Theorem \ref{thm:kuwada_dual} the heat semigroup $(P_t)_{t\geq 0}$ extends to a dual semigroup $(H_t)_{t\geq0}$ for which the contraction estimate \eqref{eq:contr_weak} holds. 
    Let $0\leq t_0 \leq t_1 \leq 1$, $\mu_0\in\Dom(\Ent)$ and $\mu_1\in\Prob(\sfG)$ be given such that $t_1\neq 0$ and let $s\mapsto \mu_s  \in \AC^2\bra*{[0,1];\Prob(\sfG)}$ be a curve joining $\mu_0$ to $\mu_1$. Since $(\Prob(\sfG),W_2)$ is a geodesic space such curves exist and can be chosen as geodesics. Using Theorem \ref{thm:ExStrongReg}, we find a family of strongly regular curves $s\mapsto \mu_s^n\in\AC\bra*{[0,1];\Prob(\sfG^{2/n})}$ for $n\in\N$ with $\mu_s^n\weakcvg\mu_s$ as $n\to\infty$. To each of these curves we apply Theorem \ref{thm:IneqReg} for $\vartheta(s) = (1-s)t_0 + st_1$. This choice is admissible since $\vartheta\in C^2([0,1])$ with $\Dot{\vartheta} = t_1-t_0 \geq 0$ and $\Ddot{\vartheta}\equiv 0$. This gives 
    \begin{align*}
        \frac{1}{2}W_2^2(H_{t_1}\mu^n_1, H_{t_0}\mu^n_0) + (t_1 - t_0) \Ent^{1/n}_\delta(H_{t_1}\mu^n_1) \leq (t_1 - t_0) \Ent^{1/n}_\delta(H_{t_0}\mu^n_0) + \frac{1}{2I_\vartheta(1)} \int_0^1\abs{\Dot{\mu}_s^n}^2\dx{s}
    \end{align*}
    for $\delta>0$. Since $\Ent^{1/n}_\delta(H_{t_1}\mu_1^n) \geq \Ent^{1/n}(H_{t_1}\mu^n_1)$ and $\lim_{\delta\to0} \Ent^{1/n}_\delta(H_{t_0}\mu^n_0) = \Ent^{1/n}(H_{t_0}\mu^n_0)$ by \eqref{eq:LimRegEntr}, it follows that
    \begin{align}\label{eq:EVIw_SRcurve}
        \frac{1}{2} W_2^2(H_{t_1}\mu^n_1, H_{t_0}\mu^n_0) + (t_1 - t_0) \Ent^{1/n}(H_{t_1}\mu^n_1) \leq (t_1 - t_0) \Ent^{1/n}(H_{t_0}\mu^n_0) + \frac{1}{2I_\vartheta(1)} \int_0^1\abs{\Dot{\mu}_s^n}^2\dx{s}.
    \end{align}
    Moreover, the contraction estimate implies that
    \begin{align*}
        W_2^2(H_{t_i}\mu_i^n, H_{t_i}\mu_i) \leq C e^{-Kt_i} W_2(\mu_i^n, \mu_i)
    \end{align*}
    holds on the extended graph $\sfG^2$ for $i=0,1$, so that $H_{t_i}\mu_i^n \weakcvg H_{t_i}\mu_i$ as $n\to\infty$. Every function $\varphi\in C(\sfG^{2/n})$ admits a continuous extension to a function on $\sfG^2$ and we also have that
    \begin{align*}
        \int_{\sfG^{2/n}} \varphi(x) \dx{H_{t_i}\mu_i^n} \to \int_\sfG \varphi\restrto{\sfG}(x) \dx{H_{t_i}\mu_i}
    \end{align*}
    as $n\to\infty$. Arguing as in the proof of Theorem \ref{thm:ExStrongReg} we obtain 
    \begin{align*}
        \Ent(H_{t_1}\mu_1) \leq \liminf\limits_{n\to\infty} \Ent^{1/n}(H_{t_1}\mu_1^n) \quad \text{and} \quad \Ent(H_{t_0}\mu_0) = \lim\limits_{n\to\infty} \Ent^{1/n}(H_{t_0}\mu_0^n).
    \end{align*}
    Additionally, the triangle inequality gives 
    \begin{align*}
        &\abs*{W_2(H_{t_1}\mu_1^n, H_{t_0}\mu_0^n) - W_2(H_{t_1}\mu_1, H_{t_0}\mu_0)} \\
        \le& \abs*{W_2(H_{t_1}\mu_1^n, H_{t_0}\mu_0^n) - W_2(H_{t_1}\mu_1^n, H_{t_0}\mu_0)} + \abs*{W_2(H_{t_1}\mu_1^n, H_{t_0}\mu_0) - W_2(H_{t_1}\mu_1, H_{t_0}\mu_0)}\\
        \leq& W_2(H_{t_0}\mu_0^n, H_{t_0}\mu_0) + W_2(H_{t_1}\mu_1^n, H_{t_1}\mu_1) 
    \end{align*}
    and therefore $W_2(H_{t_1}\mu_1^n, H_{t_0}\mu_0^n) \to W_2(H_{t_1}\mu_1, H_{t_0}\mu_0)$ as $n\to\infty$, again using the contraction estimate and the weak convergence $\mu_s^n\weakcvg\mu_s$.
    Taking the limit as $n\to\infty$ in \eqref{eq:EVIw_SRcurve} then yields
    \begin{align*}
        \frac{1}{2} W_2^2(H_{t_1}\mu_1, H_{t_0}\mu_0) + (t_1 - t_0) \Ent(H_{t_1}\mu_1) \leq (t_1 - t_0) \Ent(H_{t_0}\mu_0) + \frac{1}{2I_\vartheta(1)} \int_0^1\abs{\Dot{\mu}_s}^2\dx{s},
    \end{align*}
    which is precisely \eqref{eq:EVI_weak} because of $s\mapsto\mu_s$ being a geodesic. In the case $t_1=t_0=0$ the same arguments can be applied under additional assumption $\mu_{t_1}\in\Dom(\Ent)$, ensuring that $(t_1 - t_0)\Ent(H_{t_1}\mu_1)$ is well-defined and equal to zero.

    \textbf{\eqref{eq:EVI_weak}$\Rightarrow$\eqref{eq:BEw}:}
    This implication follows from Remark \ref{rem:EVI_to_Kuwada} together with Theorem \ref{thm:kuwada_dual}.

    \textbf{\eqref{eq:BEw}$\Rightarrow$\eqref{eq:RCD_weak}:}   
    Again, by Theorem \ref{thm:kuwada_dual} the contraction estimate is satisfied and we already showed the equivalence of \eqref{eq:BEw} to \eqref{eq:EVI_weak}. Let $t\geq0$ and $h>0$ be given and let $s\mapsto\mu_s\in\AC^2\bra*{[0,1];\Prob(\sfG)}$ be a geodesic joining $\mu_0$ and $\mu_1$. Since $\mu_0, \mu_1\in \Dom(\Ent)$, applying \eqref{eq:EVI_weak} to the pair $\mu_0\in \Dom(\Ent)$, $\mu_s\in\Prob(\sfG)$ with $t_0=t$ and $t_1=t+h$ for arbitrary $s\in[0,1]$ gives 
    \begin{align}\label{eq:WeakEVI_0s}
        \frac{1}{2}W_2^2(H_{t+h}\mu_s, H_t\mu_0) - \frac{1}{R(t,t+h)} W_2^2(\mu_s,\mu_0) \leq h\bra*{\Ent(H_t\mu_0) - \Ent(H_{t+h}\mu_s)}.
    \end{align}
    On the other hand, choosing $\mu_1\in \Dom(\Ent)$, $\mu_s\in\Prob(\sfG)$ with $t_0=t$ and $t_1=t+h$ instead yields
    \begin{align}\label{eq:WeakEVI_1s}
        \frac{1}{2}W_2^2(H{t+h}\mu_s, H_t\mu_1) - \frac{1}{R(t,t+h)} W_2^2(\mu_s,\mu_1) \leq h\bra*{\Ent(H_t\mu_1) - \Ent(H_{t+h}\mu_s)}.
    \end{align}
    Multiplying \eqref{eq:WeakEVI_0s} by $(1-s)$, \eqref{eq:WeakEVI_1s} by $s$ and adding both inequalities results in 
    \begin{align}\label{eq:SumEVIw}
        &\frac{1-s}{2}W_2^2(H_{t+h}\mu_s, H_t\mu_0) - \frac{1}{2R(t,t+h)}\bra*{(1-s) W_2^2(\mu_s,\mu_0) + sW_2^2(\mu_s,\mu_1)} + \frac{s}{2}W_2^2(H_{t+h}\mu_s, H_t\mu_1)\notag\\
        \leq& h\bra*{(1-s)\Ent(H_t\mu_0) + s\Ent(H_t\mu_1) - \Ent(H_{t+h}\mu_s)}.
    \end{align}
    Since $s\mapsto\mu_s$ is a geodesic we have that 
    \begin{align}\label{eq:Scaling_Geodesic}
        (1-s)W_2^2(\mu_s,\mu_0) + sW_2^2(\mu_s,\mu_1) &= (1-s)s^2W_2^2(\mu_0,\mu_1) + s(1-s)^2W_2(\mu_0,\mu_1) \notag\\
        &= s(1-s)W_2^2(\mu_0,\mu_1).
    \end{align}
    Further, for all $s\in[0,1]$ and $a,b\in\R$ it holds that $(1-s)a^2 + sb^2 \geq s(1-s)(a+b)^2$ and therefore
    \begin{align}\label{eq:Est_t+h_t}
        (1-s) W_2^2(H_{t+h}\mu_s, H_t\mu_0) + s W_2^2(H_{t+h}\mu_s,\mu_1) &\geq s(1-s)\bra*{W_2(H_{t+h}\mu_s, H_t\mu_0) + W_2(H_{t+h}\mu_s,\mu_1)}^2 \notag\\
        &\geq s(1-s) W_2^2(H_t\mu_0, H_t\mu_1)^2,
    \end{align}
    where we have used the triangle inequality in the second line. Combining \eqref{eq:SumEVIw} with \eqref{eq:Scaling_Geodesic} and \eqref{eq:Est_t+h_t} then gives
    \begin{align*}
        \frac{s(1-s)}{2}W_2^2(H{t}\mu_0, H_t\mu_1) - \frac{s(1-s)}{2R(t,t+h)}W_2^2(\mu_0,\mu_1) \leq h\bra*{(1-s)\Ent(H_t\mu_0) + s\Ent(H_t\mu_1) - \Ent(H_{t+h}\mu_s)}.
    \end{align*}
    Dividing by $h>0$ and rearranging the inequality above recovers \eqref{eq:RCD_weak}.

    \textbf{\eqref{eq:RCD_weak}$\Rightarrow$\eqref{eq:BEw}:}
    The contraction estimate \eqref{eq:contr_weak} holds true by assumption. This implies \eqref{eq:BEw} due to Theorem \ref{thm:kuwada_dual}.
\end{proof}

Let $\mu \ll \leb$ with $\frac{\dx{\mu}}{\dx{\leb}} =  f \in W^{1,1}(\sfG)$ be given and recall the definition of the Fisher Information \eqref{eq:FI}.
As a consequence of the chain-rule established in \cite{Erbar_2022}[Proposition 5.6], we can reformulate \ref{RCD_equiv} as follows. 

\begin{corollary}
    Let $s\mapsto\mu_s\in \AC^2\bra*{[0,1]; \Prob(\sfG)}$ be a $W_2$-geodesic connecting $\mu_0,\mu_1\in\Dom(\Ent)$ with $\mu_{s}\in\Dom(\Ent)$ for some $s\in(0,1)$. Then, for all $t\geq0$ and $h>0$ it holds that
    \begin{align*}
        \Ent(\mu_{s})& \leq (1-s) \Ent(H_t\mu_0) + s \Ent(H_t\mu_1) \\
        &+\frac{s(1-s)}{2h}\bra*{\frac{1}{R(t,t+h)} W_2^2(\mu_0,\mu_1) - W_2^2(H_t\mu_0,H_t\mu_1)} + \int_0^{t+h} \calI(H_\tau\mu_{s})\dx{\tau}.
    \end{align*}
\end{corollary}

\begin{proof}
    From \cite{Erbar_2022}[Proposition 5.6] we know that
    \begin{align*}
        \Ent(H_{t+h}\mu_{s}) - \Ent(\mu_{s}) = -\int_0^{t+h}\calI(P_\tau\mu_{s})\dx{\tau}
    \end{align*}
    and the right-hand side is finite by \cite{stefani_generalized_2022}[Lemma 4.1]. Substituting this equality into \ref{RCD_equiv} and rearranging the terms proves the claim.
\end{proof}

Let $h>0$ and define
\begin{align*}
    \omega(s) \coloneqq \frac{s(1-s)}{2h}\pra*{\frac{1}{R(0,h)}-1} W_2^2(\mu_0,\mu_1) + \int_0^{h} \calI(H_\tau\mu_{s})\dx{\tau}.
\end{align*}
Choosing $t=0$ in the corollary above gives the estimate
\begin{align*}
    \Ent(\mu_{s})& \leq (1-s) \Ent(\mu_0) + s \Ent(\mu_1) + \omega(s),    
\end{align*}
which we can think of as a distorted version of geodesic convexity. The term $\omega(s)$ quantifies how far the entropy is from being $K$-convex along geodesics. It depends on the curvature function $c(t)=Ce^{-Kt}$ through $R(0,h)$ as well as the entropy dissipation given by the Fisher Information, thus encoding information about the geometric structure of the underlying graph.

\section{Outlook}

We showed that metric graphs satisfy weak notions of lower curvature bounds with the focus on their dimensionless formulations. At present it is not known whether the more precise conditions including the dimensional parameter $N\in(0,\infty)$ admit a similar extension.  
The equivalence result gives rise to a new notion of gradient flow on these spaces, namely by $\EVIw$. We hope that in the future this concept can be applied to extend known results from $\RCD{K}$ settings to the weaker $\RCDw$-spaces. One possible application is the generalization of the results from \cite{monsaingeon_dynamical_2023} to metric graphs. Let $(X,d)$ be an $\RCD{K}$-space for $K\geq 0$. The dynamic Schrödinger problem is given by
\begin{align*}
    \inf\limits_{(\mu,v)\in{\rm CE}} \frac{1}{2}\int_0^1 \int_X \abs*{v(t,x)}^2 \dx{\mu(t,x)} \dx{t} + \frac{\beta^2}{2}\int_0^1 \int_X \abs*{\nabla \log\bra*{f(t,x)}}^2 \dx{\mu(t,x)} \dx{t}
\end{align*}
for $f = \frac{\dx{\mu}}{\dx{\leb}}$, where the infimum runs over weak solutions of the continuity equation 
\begin{align*}
    \begin{cases}
        \partial_t \mu + \nabla\cdot(v\mu) = 0 \\ \mu(0,\cdot) = \mu_0(\cdot), \quad \mu(1,\cdot) = \mu_1(\cdot)
    \end{cases}
\end{align*}
denoted by $\rm CE$.
In \cite{monsaingeon_dynamical_2023} it has been shown that the dynamic Schrödinger problem $\Gamma$-converges to the dynamic optimal transport problem
\begin{align*}
    \inf\limits_{(\mu,v)\in{\rm CE}} \frac{1}{2}\int_0^1 \int_X \abs*{v(t,x)}^2 \dx{\mu(t,x)} \dx{t}.
\end{align*}
In particular, the construction of their recovery sequence is based on \eqref{eq:EVI} and the contraction estimate \eqref{eq:ContrEst}. On metric graphs, only the weak properties \eqref{eq:EVI_weak} and \eqref{eq:contr_weak} hold true and so far it is not known whether these conditions suffice to conclude $\Gamma$-convergence.

On the other hand, the present work exclusively deals with compact metric graphs having a finite number of edges with finite length each. However, \cite{baudoin_differential_2018} proved the gradient bound 
\begin{align*}
    \sqrt{\Gamma(P_tf)} \leq Ce^{Kt}P_t\sqrt{\Gamma(f)}
\end{align*}
with $K\geq0$ on non-compact graphs, which would lead to negative curvature. Despite this result at hand, showing the equivalences in this setting remains open, the reason being the identification of the heat equation as a Wasserstein gradient flow. In \cite{Erbar_2022} the result has been shown only for compact graphs.

\section*{Acknowledgments}
The author thanks Jan-F. Pietschmann and Gianna Götzmann for valuable discussions and helpful comments and remarks, as well as Jonas Stange for bringing the reference \cite{baudoin_differential_2018} to the author’s attention.

\nocite*
\bibliographystyle{alpha}
\bibliography{literature}

\end{document}